\documentclass[12pt]{article}
\usepackage{amsfonts}
\usepackage{amssymb}
\usepackage{amsmath}
\usepackage{theorem}

\usepackage{tikz}

\usepackage{hyperref}

\usetikzlibrary{matrix}
\usetikzlibrary{arrows}

\numberwithin{equation}{section}

\input amssym.def

\makeatletter
\newcommand\qedsymbol{\hbox{$\Box$}}
\newcommand\qed{\relax\ifmmode\Box\else
  {\unskip\nobreak\hfil\penalty50\hskip1em\null\nobreak\hfil\qedsymbol
  \parfillskip=\z@\finalhyphendemerits=0\endgraf}\fi}

\newenvironment{proof}[1][{}]{\par\noindent Proof{#1}. }{\qed}

\newcommand{\hotimes}{{\,\hat{\otimes}\,}}

\newcommand{\Tree}{{\mathsf{Tree}}}

\newcommand{\pf}{{\,\pitchfork}}

\newcommand{\bfzero}{{\mathbf 0}}

\newcommand{\End}{{\mathrm{End}}}

\newcommand{\HoAlg}{{\mathsf{HoAlg}}}

\renewcommand{\c}{\circ}

\newcommand{\inv}{{\mathrm{inv}}}

\newcommand{\Cyl}{{\mathrm{Cyl}}}
\newcommand{\Cat}{{\mathrm{Cat}}}
\newcommand{\Hom}{{\mathrm{Hom}}}
\newcommand{\nod}{{\mathrm{nod}}}
\newcommand{\map}{{\mathbf{map}}}
\newcommand{\Sh}{{\mathrm{Sh}}}
\newcommand{\ari}{{\mathrm{ari}}}
\newcommand{\Gr}{{\mathrm{Gr}}}

\newcommand{\Conv}{{\mathrm{Conv}}}
\newcommand{\Cobar}{{\mathrm{Cobar}}}

\newcommand{\coDer}{{\mathrm {coDer}}}

\newcommand{\Lie}{{\mathsf{Lie}}}

\newcommand{\As}{{\mathsf{As}}}
\newcommand{\coAs}{{\mathsf{coAs}}}
\newcommand{\coLie}{{\mathsf{coLie}}}
\newcommand{\Com}{{\mathsf{Com}}}
\newcommand{\coCom}{{\mathsf{coCom}}}

\newcommand{\Ch}{{\mathsf{Ch}}}

\newcommand{\bul}{{\bullet}}
\newcommand{\al}{{\alpha}}

\newcommand{\ml}{{\mathfrak{l}}}

\newcommand{\mb}{{\mathfrak{b}}}

\newcommand{\mG}{{\mathfrak{G}}}

\newcommand{\mD}{{\mathfrak{D}}}

\newcommand{\mS}{{\mathfrak{S}}}
\newcommand{\mF}{{\mathfrak{F}}}

\newcommand{\bs}{{\bf s}}
\newcommand{\bt}{{\mathbf{t}}}

\newcommand{\bsi}{{\bf s}^{-1}\,}

\newcommand{\Om}{{\Omega}}

\newcommand{\si}{{\sigma}}
\newcommand{\ga}{{\gamma}}
\newcommand{\vf}{{\varphi}}
\newcommand{\ve}{{\varepsilon}}

\newcommand{\pa}{{\partial}}

\newcommand{\cF}{{\cal F}}

\newcommand{\cC}{{\cal C}}
\newcommand{\cK}{{\mathcal{K}}}
\newcommand{\cU}{{\mathcal{U}}}

\newcommand{\cO}{\mathcal{O}}

\newcommand{\bbk}{{\Bbbk}}

\newcommand{\bbZ}{{\mathbb Z}}

\newcommand{\D}{{\Delta}}

\newcommand{\sD}{{\mathsf{\Delta}}}

\newcommand{\wt}[1]{{\widetilde{#1}}}
\newcommand{\ti}[1]{{\tilde{#1}}}

\newcommand{\mMC}{\mathfrak{MC}}

\newcommand{\maps}{\,\colon\,}

\newcommand{\und}[1]{{\underline{#1}}}

\newcommand{\sLie}{{\mathfrak{S}}\mathsf{Lie}_{\infty}}
\newcommand{\tLie}{{\mathfrak{S}}\mathsf{Lie}_{\infty}^{\mathrm{MC}}}

\newcommand{\tensor}{\otimes}
\newcommand{\xto}[1]{\stackrel{#1}{\longrightarrow}}

\newcommand{\MC}{\mathrm{MC}}

\newcommand{\Objects}{{\mathrm{Objects}}}

\DeclareMathOperator{\id}{\mathrm{id}}

\newtheorem{thm}{Theorem}[section]
\newtheorem{defi}[thm]{Definition}

\newtheorem{lem}[thm]{Lemma}
\newtheorem{cor}[thm]{Corollary}

\newtheorem{prop}[thm]{Proposition}

\newtheorem{cond}[thm]{Condition}

\theorembodyfont{\rm}
\newtheorem{example}[thm]{Example}

\newtheorem{remark}[thm]{Remark}

\title{What do homotopy algebras form?}

\author{Vasily A. Dolgushev, Alexander E. Hoffnung, and Christopher L. Rogers}

\date{}

\begin{document}

\maketitle

\begin{abstract}
In paper \cite{EnhancedLie}, we constructed a symmetric monoidal 
category $\tLie$ whose objects are shifted (and filtered) $L_{\infty}$-algebras. 
Here, we fix a cooperad $\cC$ and show that algebras over the operad 
$\Cobar(\cC)$ naturally form a category enriched over $\tLie$. 
Following \cite{EnhancedLie}, we ``integrate'' this $\tLie$-enriched category
to a simplicial category $\HoAlg^{\sD}_{\cC}$ whose mapping spaces 
are Kan complexes. The simplicial category $\HoAlg^{\sD}_{\cC}$ gives us 
a particularly nice model of an $(\infty,1)$-category of $\Cobar(\cC)$-algebras. 
We show that the homotopy category of $\HoAlg^{\sD}_{\cC}$ is 
the localization of the category of $\Cobar(\cC)$-algebras and 
$\infty$-morphisms with respect to $\infty$ quasi-isomorphisms.  
Finally, we show that the Homotopy Transfer Theorem is 
a simple consequence of the Goldman-Millson theorem. 
\end{abstract}

\section{Introduction}
This work is motivated by Dmitry Tamarkin's answer \cite{dgcat} to Vladimir Drinfeld's 
question ``What do dg categories form?''  \cite{Drinfeld-dgcat}, and by papers \cite{Berglund}, 
\cite{DotsPoncin}, \cite{Ezra-infty}, \cite{Hinich:1997}, and \cite{Lazarev}.  
Here, we give an answer to the question ``What do homotopy algebras form?'' 

Homotopy algebras and their generalizations appear in constructions of string topology, in 
rational homotopy theory, symplectic topology, deformation quantization, and quantum field 
theory.  For a gentle introduction to this topic, we refer the reader to paper
\cite{Bruno-plus-equals}. For a more detailed exposition, a standard reference is
book \cite{LV} by Jean-Louis Loday and Bruno Vallette.    

Despite an important role of homotopy algebras, it is not clear what higher 
categorical structure stands behind the homotopy theory of homotopy algebras.
A possible way to answer to this question is to use closed model categories 
and this approach is undertaken in paper \cite{Bruno-HT} by Bruno Vallette. 
Here, we suggest a different approach which is based on the use 
of the convolution $L_{\infty}$-algebra and the Getzler-Hinich construction 
\cite{Ezra-infty}, \cite{Hinich:1997}. 

By a homotopy algebra structure on a cochain complex of
$\bbk$-vector\footnote{In this paper, we assume that
  $\textrm{char}(\bbk) = 0$.}  spaces $A$ we mean a
$\Cobar(\cC)$-algebra structure on $A$, where $\cC$ is a differential
graded (dg) cooperad satisfying some technical conditions. In this
paper, we fix such a cooperad $\cC$ and show that
$\Cobar(\cC)$-algebras form a $\tLie$-enriched category
$\HoAlg_{\cC}$, where $\tLie$ is the enhancement of the symmetric
monoidal category of shifted $L_{\infty}$-algebras introduced 
in\footnote{See also Section \ref{sec:tLie} of this paper.} \cite{EnhancedLie}.  
Then, using a generalization of the Getzler-Hinich construction
\cite{Ezra-infty}, \cite{Hinich:1997}, we show that the
$\tLie$-category of $\Cobar(\cC)$-algebras can be ``integrated'' to
the category $\HoAlg^{\sD}_{\cC}$ enriched over $\infty$-groupoids
(a.k.a. Kan complexes). 

Our approach is conceptually similar to a standard construction of the
simplicial category of chain complexes i.e., ``trivial'' homotopy
algebras. There, instead of constructing the simplicial mapping
spaces directly, one exploits the simple fact that chain complexes
are naturally enriched over chain complexes i.e., abelian
$L_{\infty}$-algebras. The simplicial category is then constructed
by first truncating the mapping complexes and applying the
Dold-Kan functor. It has been noted, for example by E.\ Getzler in \cite{Ezra-infty}, that the
integration of a $L_{\infty}$-algebra is, up to homotopy, a non-abelian
analogue of the Dold-Kan functor. Hence, 
to construct the simplicial category $\HoAlg^{\sD}_{\cC}$ of non-trivial homotopy algebras,
we proceed via analogy by first considering a category of homotopy algebras enriched
over $L_{\infty}$-algebras i.e., ``non-abelian complexes''.

Furthermore, we prove that the homotopy category of
$\HoAlg^{\sD}_{\cC}$ {\it is} the localization of the category of
$\Cobar(\cC)$-algebras with respect to $\infty$ quasi-isomorphisms, and
this is indeed the correct homotopy category of homotopy algebras (see
for example \cite[Thm.\ 11.4.12]{LV}).
Thus the simplicial category $\HoAlg^{\sD}_{\cC}$ is the sought higher
categorical structure which stands behind the homotopy category of
homotopy algebras.
 
For simplicity of the exposition, we present all constructions in the ``1-colored'' setting.
However, it is not hard to see the all the statements can be easily extended to the multi-colored 
setting.

Our paper is organized as follows. Section \ref{sec:prereq-HA} is 
devoted to prerequisites on homotopy algebras. In this section, 
we describe the convolution $L_{\infty}$-algebra $\Conv(V,A)$ associated to 
a $\cC$-coalgebra $V$ and a $\Cobar(\cC)$-algebra $A$, and recall \cite{DotsPoncin}
that, for every pair of $\Cobar(\cC)$-algebras $A,B$, Maurer-Cartan (MC) 
elements of $\Conv(\cC(A), B)$ are in bijection with $\infty$-morphisms 
from $A$ to $B$.  Section \ref{sec:HoAlg-cC} starts with a brief reminder of 
the symmetric monoidal category $\tLie$ introduced in \cite{EnhancedLie}.
Then we construct the 
$\tLie$-enriched category $\HoAlg_{\cC}$ and the corresponding simplicial 
category $\HoAlg^{\sD}_{\cC}$.  In Section \ref{sec:HT}, we show that 
$\pi_0 \big( \HoAlg^{\sD}_{\cC} \big)$ is the homotopy category of 
$\Cobar(\cC)$-algebras, i.e. every $\infty$ quasi-isomorphism of 
$\Cobar(\cC)$-algebras gives an invertible morphism in $\pi_0 \big( \HoAlg^{\sD}_{\cC} \big)$ 
and the category $\pi_0 \big( \HoAlg^{\sD}_{\cC} \big)$ has the desired universal 
property.  Also in Section \ref{sec:HT}, we characterize $\infty$
  quasi-isomorphisms as precisely those morphisms which
  induce homotopy equivalences between mapping spaces via (pre)composition.
Finally, in Section \ref{sec:HTT}, we give a very concise proof
of the Homotopy Transfer Theorem for homotopy algebras. This proof is 
based on a construction from \cite{DefHomotInvar} and a version of the 
Goldman-Millson theorem from \cite{GMtheorem}.  
 
\paragraph{On a more general definition of a homotopy algebra.} 
One could also define a homotopy algebra as an algebra over 
an operad $\cO$ which is freely generated by a collection $\cC$  
when viewed as an operad in the category of graded vector spaces, 
where the collection $\cC$ and the differential on $\cO$ are subject 
to some technical conditions. It is not hard to generalize all the constructions of 
our paper to this more general setting. However, for simplicity of
the exposition, we decided to present the whole story for the 
case when $\cO$ is the cobar construction applied to a fixed 
cooperad.

\paragraph{Is the $\tLie$-enriched category of $L_{\infty}$-algebras 
related to the rational homotopy theory?} It is not surprising that the 
answer to this question is ``yes''. However, a precise formulation of 
the answer requires some technical conditions on $L_{\infty}$-algebras 
and some amendments to the definition of the mapping space. 
So separate note \cite{RHT} is devoted to this question.

\subsubsection*{Related work on this subject}  
While this preprint was in preparation, paper \cite{KhPQ} appeared on arXiv.org. 
In this paper, the authors constructed a quasi-category of Leibniz $\infty$-algebras also 
using the Getzler-Hinich construction \cite{Ezra-infty}, \cite{Hinich:1997}. 

Our paper agrees in spirit with treatise \cite{h-alg} by Jacob Lurie. 
In this treatise, J. Lurie develops a systematic approach to algebraic 
structures on objects in a fixed symmetric monoidal $\infty$-category. 
This approach allows him to define $E_{\infty}$-ring as a commutative algebra object 
in the $\infty$-category of spectra. Even though, the framework of the 
presentation of \cite{h-alg} is very general, we could not find a particular 
statement in the current version of draft \cite{h-alg} from which all statements 
proved in our paper will follow. One of the reasons for this is that, in \cite{h-alg}, 
J. Lurie usually assumes the all the (co)chain complexes under consideration
satisfy some kind of boundedness condition whereas in our paper we 
do not impose this assumption. 
  
In paper \cite{Bruno-HT}, Bruno Vallette used the simplicial localization 
methods of Dwyer and Kan \cite{DK} to upgrade the category of
$\Cobar(\cC)$-algebras with $\infty$-morphisms (for a cooperad $\cC$ 
satisfying some conditions) to a simplicial category (see Theorem 3.5 in \cite{Bruno-HT}). 
The homotopy category of this simplicial category is isomorphic to that of 
$\HoAlg^{\sD}_{\cC}$ and we believe that there is a more precise link between these
two simplicial categories.

~\\

\noindent
\textbf{Acknowledgements:} We would like to thank Thomas Willwacher
for useful discussions, and Jim Stasheff for helpful comments. 
V.A.D. and C.L.R. acknowledge NSF grant DMS-1161867 
for a partial support. C.L.R. also acknowledges support from the German Research Foundation 
(Deutsche Forschungsgemeinschaft (DFG)) through the Institutional Strategy of the
University of G\"ottingen. Finally, we would like to thank an anonymous referee for reading carefully 
our manuscript and for many helpful suggestions.

\subsubsection*{Notation and conventions} 
A big part of our conventions is borrowed from \cite{EnhancedLie}. 
For example, the ground field $\bbk$ has characteristic zero and
$\Ch_{\bbk}$ denotes the category of unbounded cochain complexes 
of $\bbk$-vector spaces.  Any $\bbZ$-graded vector space $V$ is tacitly considered 
as the cochain complex with the zero differential. Following 
\cite{EnhancedLie}  we frequently use the ubiquitous
abbreviation ``dg'' (differential graded) to refer to algebraic objects 
in $\Ch_{\bbk}$\,. The notation $\bs V$ (resp. by $\bsi V$) is reserved 
for the suspension (resp. the 
desuspension) of a cochain complex $V$\,, i.e.
$\big(\bs V\big)^{\bul} = V^{\bul-1}$ and $\big(\bs^{-1} V\big)^{\bul} = V^{\bul+1}$\,.
Finally, for a pair $V$, $W$ of $\bbZ$-graded vector spaces we denote by 
$$
\Hom (V,W)
$$
the corresponding inner-hom object in the category of $\bbZ$-graded vector spaces.

Following \cite{EnhancedLie} we tacitly assume the Koszul sign rule. In particular,   
$$
(-1)^{\ve(\si; v_1, \dots, v_m)}
$$ 
will always denote the sign factor corresponding to the permutation $\si \in S_m$ of 
homogeneous vectors $v_1, v_2, \dots, v_m$. Namely, 
\begin{equation}
\label{ve-si-vvv}
(-1)^{\ve(\si; v_1, \dots, v_m)} := \prod_{(i < j)} (-1)^{|v_i | |v_j|}\,,
\end{equation}
where the product is taken over all inversions $(i < j)$ of $\si \in S_m$.

For a finite group $G$ acting on a cochain complex
(or a graded vector space) $V$, we denote by 
$$
V^G,  \qquad \textrm{ and } \qquad  V_G,
$$
respectively, the subcomplex of $G$-invariants in $V$ and
the quotient complex of $G$-coinvariants. Using the advantage of 
the zero characteristic, we often identify $V_G$ with $V^G$ 
via this isomorphism
\begin{equation}
\label{coinvar-invar}
v ~\mapsto~ \sum_{g \in G} g (v) \maps V_G \to V^G\,. 
\end{equation}

For a graded vector space (or a cochain complex) $V$
the notation $S(V)$ (resp. $\und{S}(V)$) is reserved for the
underlying vector space of the
symmetric algebra (resp. the truncated symmetric algebra) of $V$: 
$$
S(V) = \bbk \oplus V \oplus S^2(V) \oplus S^3(V) \oplus \dots\,, 
$$ 
$$
\und{S}(V) =  V \oplus S^2(V) \oplus S^3(V) \oplus \dots\,,
$$
where 
$$
S^n(V) = \big( V^{\otimes_{\bbk}\, n} \big)_{S_n}\,.
$$

We denote by $\As$, $\Com$, $\Lie$ the operads governing 
associative, commutative (and associative), and Lie algebras, 
respectively. We set $\As(0) = \Com(0) =\bfzero$\,. In other words, algebras 
over $\As$ and $\Com$ are non-unital. We denote by $\As_{\infty}$, 
$\Com_{\infty}$, and $\Lie_{\infty}$ the dg operads which 
govern homotopy versions of the corresponding algebras. 
Furthermore, we denote by $\coAs$, $\coCom$, and $\coLie$, 
the cooperads which are obtained from $\As$, $\Com$, and $\Lie$ respectively, 
by taking the linear dual. 

For an augmented operad $\cO$,  
we denote by $\cO_{\c}$ the kernel of the augmentation.
Dually, for a coaugmented cooperad $\cC$ 
we denote by $\cC_{\c}$ the cokernel of the coaugmentation.
Recall that for every augmented operad $\cO$ (resp.  coaugmented cooperad $\cC$)
the collection  $\cO_{\c}$ (resp. $\cC_{\c}$) is naturally a pseudo-operad 
(resp. pseudo-cooperad) in the sense of \cite[Section 3.2]{notes} (resp. 
\cite[Section 3.4]{notes}). 

For an operad (resp. a cooperad) $P$ and a cochain complex $V$ we denote by 
$P(V)$ the free $P$-algebra (resp. the cofree\footnote{In this paper we only consider 
nilpotent coalgebras.} $P$-coalgebra) generated by $V$: 
\begin{equation}
\label{P-Schur-V}
P(V) : = \bigoplus_{n \ge 0} \Big( P(n) \otimes V^{\otimes \, n} \Big)_{S_n}\,.
\end{equation}
For example,
$$
\Com(V) = \und{S}(V) \qquad \textrm{and} \qquad \coCom(V) = \und{S}(V)\,.
$$

For a cooperad $\cC$, it is sometimes more convenient to 
work with the cofree $\cC$-coalgebra defined as the direct sum 
of the space of invariants (instead of coinvariants): 
\begin{equation}
\label{cC-V-invar}
\bigoplus_{n \ge 0} \Big( \cC(n) \otimes V^{\otimes \, n} \Big)^{S_n}\,.
\end{equation}
For example, it is more natural to define a $\cC$-coalgebra 
structure on a graded vector space (or a cochain complex) $V$ as 
a collection of comultiplication maps 
\begin{equation}
\label{Delta-n}
\D_n : V \to \Big( \cC(n) \otimes V^{\otimes \, n} \Big)^{S_n} 
\end{equation}
satisfying some natural coassociativity axioms which are 
obtained by dualizing the corresponding axioms for algebras 
over an operad.

We denote this direct sum by  
\begin{equation}
\label{cC-V-inv}
\cC(V)^{\inv} : = \bigoplus_{n \ge 0} \Big( \cC(n) \otimes V^{\otimes \, n} \Big)^{S_n}
\end{equation}
and keep in mind that $\cC(V)^{\inv}$ is isomorphic to $\cC(V)$ via map \eqref{coinvar-invar}. 

Following \cite{EnhancedLie}, $\mS$ and $\mS^{-1}$ denote the underlying 
collections of the operads 
$$
\End_{\bsi \bbk}\,, \qquad \textrm{and} \qquad \End_{\bs \bbk}\,,
$$ 
respectively. Furthermore, for a dg (co)operad $P$, we denote 
by $\mS P$ (resp. $\mS^{-1} P$) the dg (co)operad which is obtained from 
$P$ by tensoring with $\mS$ (resp. $\mS^{-1}$): 
$$
\mS P : = \mS \otimes P\,, \qquad   
\mS^{-1} P : = \mS^{-1} \otimes P\,.
$$
For example, $\sLie$-algebras are algebra over the dg operad
\begin{equation}
\label{sLie-oper}
\sLie : =\Cobar(\coCom)
\end{equation}
and $L_{\infty}$-algebras are algebras over the dg operad 
\begin{equation}
\label{Lie-inf-oper}
\Lie_{\infty} : =\Cobar(\mS^{-1} \coCom)\,.
\end{equation}

Just as in \cite{EnhancedLie},
we often call $\sLie$-algebras {\it shifted $L_{\infty}$-algebras}. 
Although a $\sLie$-algebra structure on a cochain complex $V$ is the same thing 
as an $L_{\infty}$ structure on $\bs V$\,, working with $\sLie$-algebras has 
important technical advantages. This is why we prefer to deal with shifted 
$L_{\infty}$ structures on $V$ versus original $L_{\infty}$ structures on $\bs V$. 

Following \cite{EnhancedLie}, we denote the tensor product of $\infty$-morphisms 
of $\sLie$-algebras by $\otimes$ even though the tensor product of the 
$\sLie$-algebras is denoted by $\oplus$.  

We often use the plain arrow $\to$ for $\infty$-morphisms of homotopy algebras. 
Of course, it should be kept in mind that in general such morphisms are maps of 
the corresponding coalgebras but not the underlying cochain complexes.     

The abbreviation ``MC'' is reserved for the term ``Maurer-Cartan''.

\paragraph{Conventions about trees.} By a {\it tree} we mean
a connected graph without cycles with a marked vertex 
called {\it the root}.  In this paper, we assume that the root of 
every tree has valency $1$ (such trees are sometimes called  {\it planted}). 
The edge adjacent to the root is called the {\it root edge}.
Non-root vertices of valency $1$ are 
called {\it leaves}.  A vertex is called {\it internal} if it is 
neither a root nor a leaf. We always orient trees in the 
direction towards the root. Thus every internal vertex 
has at least $1$ incoming edge and exactly $1$ outgoing edge. 
An edge adjacent to a leaf is called {\it external}.   
A tree $\bt$ is called {\it planar} if, for every internal vertex $v$ of $\bt$, the set 
of edges terminating at $v$ carries a total order.

Let us recall \cite[Section 2]{notes} that for every  planar tree $\bt$ 
the set $V(\bt)$ of all its vertices is equipped with a natural total order
such that the root is the smallest vertex of the tree.  
 
For a non-negative integer $n$, an $n$-labeled planar tree $\bt$ is 
a planar tree equipped with 
an injective map 
\begin{equation}
\label{labeling}
\ml : \{1,2, \dots, n\} \to L(\bt)
\end{equation}
from the set $\{1,2, \dots, n\}$ to the set $L(\bt)$ of leaves of $\bt$\,.
Although the set $L(\bt)$ has a natural total order we do not require 
that map \eqref{labeling} is monotonous. 

The set $L(\bt)$ of leaves of an $n$-labeled planar tree $\bt$
splits into the disjoint union of the image $\ml(\{1,2, \dots, n\})$
and its complement. We call leaves in the image of $\ml$
{\it labeled}.

A vertex $x$ of an $n$-labeled planar tree $\bt$ is called 
{\it nodal} if it is neither the root, nor a labeled leaf. 
We denote by $V_{\nod}(\bt)$ the set of all nodal vertices of 
$\bt$. Keeping in mind the canonical total order on 
the set of all vertices of $\bt$ we can say things like
``the first nodal vertex'', ``the second nodal vertex'', and
``the $i$-th nodal vertex''. 

It is convenient to talk about (co)operads and pseudo-(co)operads 
using the groupoid  $\Tree(n)$ of $n$-labeled planar trees. 
Objects of $\Tree(n)$ are $n$-labeled planar trees and 
morphisms are \und{non-planar} isomorphisms of the corresponding 
(non-planar) trees compatible with labeling. 

Following \cite[Section 3.2, 3.4]{notes}, for an $n$-labelled planar tree $\bt$ and 
pseudo-operad $P$ (resp. pseudo-cooperad $Q$) the notation
$\mu_{\bt}$ (resp. the notation $\D_{\bt}$) is reserved for the multiplication 
map 
\begin{equation}
\label{mu-bt}
\mu_{\bt} :  P (r_1) \otimes P(r_2) \otimes \dots \otimes P(r_k) \to P(n)
\end{equation}
and the comultiplication map 
\begin{equation}
\label{Delta-bt}
\D_{\bt} : Q(n) \to Q(r_1) \otimes Q(r_2) \otimes \dots \otimes Q(r_k)\,.
\end{equation}
respectively. Here, $k$ is the number of nodal vertices of the planar tree $\bt$
and $r_i$ is the number of edges (of $\bt$) which terminate at the 
$i$-th nodal vertex of $\bt$\,. 

For example, if $\bt_{n,k,i}$ is the labeled planar tree shown 
on figure\footnote{On figures, small white circles are used for 
nodal vertices and small black circles are used for 
all the remaining vertices.} \ref{fig:bt-n-k-i} then the map 
\begin{equation}
\label{mu-bt-nki}
\mu_{\bt_{n,k,i}} :  P (n) \otimes P(k) \to P(n+k-1)
\end{equation}
is precisely the $i$-th elementary insertion
\begin{equation}
\label{circ-i}
\circ_i :   P (n) \otimes P(k) \to P(n+k-1)\,. 
\end{equation}
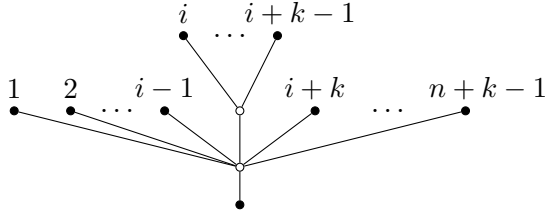
\begin{figure}[htp]
\centering 
\begin{tikzpicture}[scale=0.5]
\tikzstyle{w}=[circle, draw, minimum size=3, inner sep=1]
\tikzstyle{b}=[circle, draw, fill, minimum size=3, inner sep=1]
\node[b] (v1) at (0, 2) {};
\draw (0,2.6) node[anchor=center] {{\small $1$}};
\node[b] (v2) at (1.5, 2) {};
\draw (1.5,2.6) node[anchor=center] {{\small $2$}};
\draw (2.8,2) node[anchor=center] {{\small $\dots$}};
\node[b] (v1i) at (4, 2) {};
\draw (4,2.6) node[anchor=center] {{\small $i-1$}};
\node[w] (vv) at (6, 2) {};
\node[b] (vi) at (4.5, 4) {};
\draw (4.5,4.6) node[anchor=center] {{\small $i$}};
\draw (5.8,4) node[anchor=center] {{\small $\dots$}};
\node[b] (v1ik) at (7, 4) {};
\draw (7.6,4.6) node[anchor=center] {{\small $i+k-1$}};
\node[b] (vik) at (8, 2) {};
\draw (8,2.6) node[anchor=center] {{\small $i+k$}};
\draw (10,2) node[anchor=center] {{\small $\dots$}};
\node[b] (v1nk) at (12, 2) {};
\draw (12.6,2.6) node[anchor=center] {{\small $n+k-1$}};
\node[w] (v) at (6, 0.5) {};
\node[b] (r) at (6, -0.5) {};
\draw (vv) edge (vi);
\draw (vv) edge (v1ik);
\draw (v) edge (v1);
\draw (v) edge (v2);
\draw (v) edge (v1i);
\draw (v) edge (vv);
\draw (v) edge (vik);
\draw (v) edge (v1nk);
\draw (r) edge (v);
\end{tikzpicture}
\caption{The $(n+k-1)$-labeled planar tree $\bt_{n,k,i}$} 
\label{fig:bt-n-k-i}
\end{figure}

Recall that a (co)operad $P$ is {\it reduced} if it 
satisfies this technical condition 
\begin{equation}
\label{reduced}
P(0)= \bfzero\,.
\end{equation}

When we deal with reduced (co)operads, we may discard all labeled trees which have 
at least one nodal vertex with no incoming edges. In other words, one 
 may safely assume that nodal vertices of a labeled tree 
are precisely its internal vertices, i.e. map \eqref{labeling} is a bijection.

\section{Prerequisites on homotopy algebras}
\label{sec:prereq-HA}

Let $\cC$ be a dg coaugmented cooperad satisfying the following technical condition: 
\begin{cond}
\label{cond:cC-circ-filtered}
The pseudo-cooperad $\cC_{\c}$ carries an ascending filtration 
\begin{equation}
\label{cC-circ-filtr}
\bfzero =  \cF^0 \cC_{\c} \subset \cF^1 \cC_{\c} \subset  \cF^2 \cC_{\c} \subset  \cF^3 \cC_{\c} \subset \dots 
\end{equation}
which is compatible with the  pseudo-cooperad  structure in the following sense: 
\begin{equation}
\label{D-bt-filtr}
\D_{\bt} \big( \cF^m \cC_{\c} (n) \big) \subset 
\bigoplus_{m_1 + \dots + m_k  = m} 
\cF^{m_1} \cC_{\c} (r_1) \otimes  \cF^{m_2} \cC_{\c} (r_2) \otimes  \dots
\otimes  \cF^{m_k} \cC_{\c} (r_k)\,,
\end{equation}
$$
 \bt \in \Tree(n)\,,
$$
where $k$ is the number of nodal vertices of the planar tree $\bt$
and $r_i$ is the number of edges (of $\bt$) which terminate at the 
$i$-th nodal vertex of $\bt$\,. We also assume that $\cC_{\c}$ is cocomplete 
with respect to filtration \eqref{cC-circ-filtr}, i.e.
\begin{equation}
\label{cocomplete}
\cC_{\c} = \bigcup_{m} \cF^m \cC_{\c}\,. 
\end{equation}
\end{cond}

We will use the following pedestrian definition of 
homotopy algebras of a given type: 
\begin{defi}
\label{dfn:infinity}
Homotopy algebras of type $\cC$ are 
algebras in $\Ch_{\bbk}$ over the operad 
$\Cobar(\cC)$\,.  
\end{defi}
Thus, $A_{\infty}$-, $L_{\infty}$-, and
$\Com_{\infty}$-algebras are examples of 
homotopy algebras. Indeed,  $A_{\infty}$-algebras
are algebras over the operad $\Cobar(\mS^{-1} \coAs)$, 
$L_{\infty}$-algebras are algebras over \\ $\Cobar(\mS^{-1}\coCom)$, 
and $\Com_{\infty}$-algebras are algebras over $\Cobar(\mS^{-1} \coLie)$\,.

Let $A$ be a cochain complex and let 
\begin{equation}
\label{coDer-cC-A}
\coDer\big( \cC(A) \big) 
\end{equation}
be the dg Lie algebra of coderivations of the $\cC$-coalgebra $\cC(A)$\,.

It is known \cite[Corollary 5.3]{notes}, \cite[Proposition 2.15]{GJ}, 
that $\Cobar(\cC)$-algebra structures on $A$ are in bijection 
with MC elements of the dg Lie subalgebra 
\begin{equation}
\label{coDer-pr-cC-A}
\coDer'\big( \cC(A) \big)  \subset \coDer\big( \cC(A) \big) 
\end{equation}
of coderivations $Q$ satisfying the condition 
\begin{equation}
\label{coder-condition}
Q \Big|_{A} = 0\,.
\end{equation}

Thus every homotopy algebra $A$ of type $\cC$ gives 
us a (dg) $\cC$-coalgebra
\begin{equation}
\label{cC}
\Big( \cC(A), \pa + Q  \Big)
\end{equation}
where $\pa$ is the differential on $\cC(A)$ induced
by the ones on $A$ and $\cC$\,.

This observation is used to define a notion of 
$\infty$-morphism of homotopy algebras of type $\cC$\,.
In other words, 
\begin{defi}
\label{dfn:morphism}
Let $A$ and $B$ be homotopy algebras of type $\cC$
and let $Q_A$ (resp. $Q_B$) be the MC element of 
$\coDer(\cC(A))$ (resp. $\coDer(\cC(B))$) corresponding 
to the homotopy algebra structure on $A$ (resp. on $B$).
Then, an $\infty$-morphism from $A$ to $B$ is 
a homomorphism of dg $\cC$-coalgebras 
\begin{equation}
\label{U-morph}
U ~:~ \Big( \cC(A), \pa + Q_A  \Big) ~\to~ \Big( \cC(B), \pa + Q_B  \Big)\,.
\end{equation}
\end{defi}
Recall that any such homomorphism $U$ is uniquely determined 
by its composition 
\begin{equation}
\label{U-pr}
U' : = p_B \circ U  
\end{equation}
with the canonical projection $p_B : \cC(B) \to B$. 
In this paper, we often use this convention: {\it $U'$ denotes composition 
\eqref{U-pr} corresponding to a homomorphism of
$\cC$-coalgebras $U$ \eqref{U-morph}.}

Given an $\infty$-morphism $U$ from $A_1$ to $A_2$ and 
an $\infty$-morphism $\ti{U}$ from $A_2$ to $A_3$, their composition 
is defined, in the obvious way, as the composition of the corresponding 
homomorphisms of dg $\cC$-coalgebras. We denote by 
$$
\Cat_{\cC}
$$
the category whose objects are homotopy algebras of type $\cC$ (a.k.a. $\Cobar(\cC)$-algebras) 
and whose morphisms are  $\infty$-morphisms with the above obvious composition.

\begin{remark}
\label{rem:why-homot-alg}
Due to \cite[Proposition 38]{MVnado11}, Condition \ref{cond:cC-circ-filtered}
implies that $\Cobar(\cC)$ is a cofibrant object in the closed model category 
of dg operads. The same condition also guarantees
that homotopy algebras of type $\cC$ enjoy the obvious version of 
the Homotopy Transfer Theorem (see \cite[Theorem 10.3.2]{LV}). 
A very concise proof of this important theorem in 
given in Section \ref{sec:HTT} of this paper.   
\end{remark}

\begin{remark}
\label{rem:Markl}
Another way to define the notion of $\infty$-morphism 
of homotopy algebras is to resolve a $2$-colored operad which
governs pairs of algebras with a morphism between them.   
This different approach to the ``higher category'' of 
homotopy algebras is initiated in works \cite{Doubek}, \cite{Markl-h-alg} 
of M. Doubek and M. Markl. 
\end{remark}

\subsection{The convolution $\sLie$-algebra}

Let $V$ be a $\cC$-coalgebra and  $A$ be a homotopy 
algebra of type $\cC$ (i.e. an algebra over $\Cobar(\cC)$)\,.

On the graded vector space
\begin{equation}
\label{Conv-V-A}
\Hom(V, A)
\end{equation}
we define the following multi-brackets: 
\begin{equation}
\label{1-bracket}
\{f\}(v) = \pa_{A}\, f(v) -(-1)^{|f|} f(\pa_{V}\, v)  +  p_A  \circ Q_A\big( 1 \otimes f(\D_1(v)) \big)
\end{equation}
\begin{equation}
\label{m-bracket}
\{f_1, \dots, f_m\} (v) = 
p_A \circ Q_A \big( 1 \otimes f_{1}\otimes \dots \otimes f_{m}  
(\D_m (v))  \big)\,, \qquad m \ge 2\,,
\end{equation}
where $\D_m$ is the $m$-th component of the comultiplication 
$$
\D_m : V \to \Big( \cC(m) \otimes V^{\otimes\, m} \Big)^{S_m}
$$
and $p_A$ is the canonical projection 
$$
p_A : \cC(A) \to A\,.
$$

Note that, since $Q_A$ has degree $1$, each multi-bracket in \eqref{m-bracket} 
also carries degree $1$\,. 

We claim that 
\begin{prop}
\label{prop:Conv-V-A}
For every $\cC$-coalgebra $V$ and a $\Cobar(\cC)$-algebra $A$, 
multi-brackets \eqref{1-bracket}, \eqref{m-bracket} equip the graded vector 
space  $\Hom(V, A)$ with a structure of a $\sLie$-algebra.
\end{prop}
The proof of this proposition is given in Appendix \ref{app:proof-Conv}. 

\begin{defi}
\label{dfn:Conv}
Let $V$ be a $\cC$-coalgebra and  $A$ be a homotopy 
algebra of type $\cC$\,. Then $\sLie$-algebra 
\eqref{Conv-V-A} is called the convolution algebra of the pair 
$(V, A)$\,. We use the notation:
$$
\Conv(V,A) : = \Hom(V, A)\,.
$$
\end{defi}

\subsubsection{Convolution $\sLie$-algebra and $\infty$-morphisms}

For a pair $A$, $B$ of  homotopy algebras of type $\cC$, we 
consider the convolution $\sLie$-algebra
\begin{equation}
\label{Conv-cCA-B}
L = \Hom(\cC(A), B)\,,
\end{equation}
where the $\cC$-coalgebra $\cC(A)$ is considered with 
the differential $\pa + Q_{A}$\,, and $\pa$ comes from the 
differential on $A$ and the differential on $\cC$.  

We observe that the $\sLie$-algebra carries the following 
descending filtration
\begin{equation} 
\label{Conv-filter}
\begin{split}
\cF^{\ari}_0 L & \supset \cF^{\ari}_{1} L \supset \cF^{\ari}_{2}L \supset \cdots  \\[0.3cm]
\cF^{\ari}_{n}L & = \{ f  \in \Hom(\cC(A), B) ~\vert ~ 
f \big|_{\cC(m) \otimes_{S_m} A^{\otimes \, m}} =0
~~ \forall ~ m < n  \}.
\end{split}
\end{equation}

It is also easy to check that: 
\begin{prop} 
\label{prop:filtered}
The convolution $\sLie$-algebra structure given by \eqref{1-bracket} and \eqref{m-bracket}
is compatible with filtration \eqref{Conv-filter} i.e.\
\[
\Bigl \{\cF^{\ari}_{i_{1}}L,\cF^{\ari}_{i_{2}}L,\ldots,\cF^{\ari}_{i_{k}}L \Bigr \} \subseteq
\cF^{\ari}_{i_{1} + i_{2} + \cdots + i_{k}} L \quad \forall~~ k >1, 
\]
Moreover, the $\sLie$-algebra $L = \Hom(\cC(A), B)$ is complete 
with respect to this filtration, i.e. 
$$
L =  \varprojlim_{k} L / \cF^{\ari}_{k}L\,. 
$$ 
\qed
\end{prop}

The notion of the convolution $\sLie$-algebra is partially justified by 
the following lemma: 
\begin{lem} 
\label{lem:infty-morph}
Let $A$ and $B$ be homotopy algebras of type $\cC$. If the cooperad 
$\cC$ satisfies condition 
\begin{equation}
\label{cC-reduced}
\cC(0) = \bfzero
\end{equation}
then 
\begin{equation}
\label{Conv-AB-pronilp}
\Hom(\cC(A), B) = \cF^{\ari}_1 \Hom(\cC(A), B)\,. 
\end{equation}
In particular, the $\sLie$-algebra $\Hom(\cC(A), B)$ is pro-nilpotent. 
Furthermore, the assignment 
$$
U ~\mapsto~ U' : = p_{B} \circ U
$$
is a bijection between the set of $\infty$-morphisms from $A$ to $B$ and
the set of MC elements of the $\sLie$-algebra $\Hom(\cC(A), B)$\,.
\end{lem}
\begin{proof}
The first statement of the lemma follows directly from 
the definition of filtration \eqref{Conv-filter} on the 
$\sLie$-algebra $\Hom(\cC(A), B)$. Since  $\Hom(\cC(A), B)$ is complete 
with respect to this filtration, we conclude that the $\sLie$-algebra $\Hom(\cC(A), B)$
is pronilpotent. 

This conclusion allows us to write the MC equation 
in the $\sLie$-algebra $\Hom(\cC(A), B)$ for any degree zero element. 

According to Definition \ref{dfn:morphism}, an  $\infty$-morphism from 
$A$ to $B$ is a homomorphism of dg $\cC$-coalgebras 
$$
U ~:~ \Big( \cC(A), \pa + Q_A  \Big) ~\to~ \Big( \cC(B), \pa + Q_B  \Big)\,.
$$ 
Since the $\cC$-coalgebra $\cC(B)$ is cofree (over graded vector spaces), 
the homomorphism $U$ is uniquely determined by its composition 
\begin{equation}
\label{pB-U}
U' : = p_{B} \circ U :   \cC(A) \to B\,.
\end{equation}
Furthermore, the compatibility of $U$ with the differentials $\pa  + Q_A$
and $\pa + Q_B$ is equivalent to 
the equation 
\begin{multline}
\label{U-and-diffs}
\pa \circ U' (X; a_1, \dots, a_m) \\ 
 + Q_B \circ (1 \otimes U') \circ \D_1 (X; a_1, \dots, a_m) -  U' \circ (\pa + Q_A) (X; a_1, \dots, a_m)    \\
+ \sum_{k=2}^{\infty} \frac{1}{k!} Q_B \circ (1 \otimes (U')^{\otimes k})  
\circ \D_k (X; a_1, \dots, a_m) = 0\,,    
\end{multline}
where $(X; a_1, \dots, a_m)$ represents a vector in $\cC(A)$ and the factor 
$1/k!$ in the last sum comes from the identification of $\cC(B)^{\inv}$ with 
$\cC(B)$ via the inverse of isomorphism \eqref{coinvar-invar}. 

Using the definition of multi-brackets \eqref{1-bracket}, \eqref{m-bracket} on  
$\Hom(\cC(A), B)$, we see that \eqref{U-and-diffs} is precisely the MC equation 
for $U'$ in the $\sLie$-algebra $\Hom(\cC(A), B)$\,.
Thus 
$$
U ~~ \leftrightarrow ~~ U' : =  p_{B} \circ U
$$
is a desired bijection between the set of $\infty$-morphisms from $A$ to $B$
and the set of MC elements of the $\sLie$-algebra $\Hom(\cC(A), B)$\,.
\end{proof}
\begin{remark}
\label{rem:Dots-Poncin}
A version of Lemma \ref{lem:infty-morph} is proved in \cite{DotsPoncin}.
See Proposition 3 in \cite[Section 1.3]{DotsPoncin}.
\end{remark}
\begin{example}
\label{exam:A-infinity}
Let us recall that $\coAs$ is the cooperad which governs coassociative coalgebras 
without counit and the dg operad $\Cobar(\mS^{-1} \coAs)$ governs (flat) $A_{\infty}$-algebras. 
It is easy to see that, for every $A_{\infty}$-algebra $A$,
\begin{equation}
\label{Conv-cC-A-A}
\Hom(\mS^{-1} \coAs(A),  A )
\end{equation}
is the completed version of the truncated Hochschild cochain complex of $A$
\begin{equation}
\label{Hoch-A}
\prod_{n \ge 1} \bs^{n-1} \Hom(A^{\otimes \, n}, A)
\end{equation}
and the shifted $L_{\infty}$-algebra on \eqref{Hoch-A} is obtain by 
symmetrizing the cup product and its higher analogues. It is worthy of mentioning that 
the Hochschild differential on \eqref{Hoch-A} is obtained by twisting the differential 
on \eqref{Conv-cC-A-A} by the MC element corresponding to the identity map $\id: A \to A$. 
\end{example}

\section{The $\tLie$-enriched category $\HoAlg_{\cC}$ and the simplicial category  $\HoAlg^{\sD}_{\cC}$}
\label{sec:HoAlg-cC}

\subsection{A brief reminder of the symmetric monoidal category $\tLie$}
\label{sec:tLie}

Let us recall \cite{EnhancedLie} that a $\sLie$-algebra $(L, \pa, \{\cdot,\cdot\},
\{\cdot,\cdot,\cdot\},\ldots)$ is \emph{filtered} if the underlying
complex $(L, \pa)$ is equipped with a complete descending filtration,
\begin{equation}
\label{filtr-L}
L = \cF_{1}L \supset \cF_{2}L \supset  \cF_{3}L  \cdots
\end{equation}
\begin{equation}
\label{L-complete}
L =\varprojlim_{k} L/\cF_{k}L\,,
\end{equation}
which is compatible with the brackets, i.e.
\[
\Bigl \{\cF_{i_{1}}L,\cF_{i_{2}}L,\ldots,\cF_{i_{m}}L \Bigr \} \subseteq
\cF_{i_{1} + i_{2} + \cdots + i_{m}} L \quad \forall ~~ m >1.
\]

For example, if $A$ and $B$ are $\Cobar(\cC)$-algebras and the cooperad 
$\cC$ satisfies the condition $\cC(0) = \bfzero$, then the filtration ``by arity''
\eqref{Conv-filter} on $\Hom(\cC(A), B)$ satisfies the above conditions. 
In other words,  $\Hom(\cC(A), B)$ is a filtered $\sLie$-algebra. 

The filtration \eqref{filtr-L} induces a natural descending 
filtration and hence a topology on $\und{S}(L)$. Just as in \cite{EnhancedLie}, 
we tacitly assume that $\infty$-morphisms of filtered $\sLie$-algebras 
are continuous with respect to this topology. 

Let us also recall \cite{EnhancedLie} that an \emph{enhanced morphism} 
\[
L_1 \xto{(\al, F)} L_2
\]
between filtered $\sLie$-algebras is a pair consisting of a MC element
$\alpha \in L_2$ and a (continuous) $\infty$-morphism
$F \maps L_1 \to L^{\alpha}_2$, where $L^{\alpha}_2$ is obtained from 
$L_2$ via twisting\footnote{See \cite[Section 2]{EnhancedLie} for details on twisting 
by MC elements.} by the MC element $\al$. 
The composition of two enhanced morphisms $L_1 \xto{(\al_2, F)} L_2$
and $L_2 \xto{(\al_3, G)} L_3$ is the pair
\begin{equation}
\label{composition}
(\al_3 + G_*(\al_2), G^{\al_2} \circ F)\,, 
\end{equation}
where the $\infty$-morphism $G^{\al_2}$ is obtained from 
$G$ via twisting by the MC element $\al_2 \in L_2$.

Following \cite{EnhancedLie}, we denote by $\tLie$ the 
category of filtered $\sLie$-algebras with the above enhanced 
morphisms. 

Given two filtered $\sLie$ algebras $(L, \{\cdot,\cdot\},
\{\cdot,\cdot, \cdot\},\ldots)$ and $(\wt{L}, \{\cdot,\cdot \}\wt{~},
\{\cdot,\cdot, \cdot\}\wt{~},\ldots)$, one obtains a
filtered $\sLie$ structure on the direct sum $L \oplus \wt{L}$ 
by setting
\[
\{x_1 + x'_1, x_2 + x'_2,\ldots, x_k + x'_k\} :=
\{x_1, x_2,\ldots, x_k\} + \{x'_1, x'_2,\ldots, x'_k\}\wt{~}\,,
\] 
and 
$$
\cF_k  (L \oplus \wt{L}) : = (\cF_k L) \oplus  (\cF_k \wt{L})\,.
$$

If $\al$ and $\wt{\al}$ are MC elements of $L$ and $\wt{L}$, respectively, 
then $\al + \wt{\al} \in L \oplus \wt{L}$ is clearly a MC element of 
the $\sLie$-algebra $L\oplus \wt{L}$. Furthermore, the operation of 
twisting (by a MC element) is compatible with $\oplus$, i.e. the $\sLie$-algebra
$L^{\al} \oplus \wt{L}^{\wt{\al}}$ is canonically isomorphic to the $\sLie$-algebra 
$(L\oplus \wt{L})^{\al + \wt{\al}}$\,.

Using these observations, we show, in \cite[Section 3.1]{EnhancedLie}, that 
the assignment 
$$
(L, \wt{L}) ~\mapsto~ L \oplus \wt{L}
$$
can be upgraded to a structure of a symmetric monoidal category on $\tLie$ 
with $\bfzero$ being the unit object.

\subsection{The $\tLie$-enriched category of $\Cobar(\cC)$-algebras}
Given two $\Cobar(\cC)$-algebras $A$ and $B$\,, we denote by 
\begin{equation}
\label{Map-A-B}
\map(A, B) : =  \Hom(\cC(A), B)
\end{equation}
the convolution $\sLie$-algebra defined in Proposition \ref{prop:Conv-V-A}. 

In this section, we construct a $\tLie$-enriched category \cite{EnhancedLie}, \cite{Kelly} $\HoAlg_{\cC}$
whose objects are homotopy algebras $A,B, \dots$ of type $\cC$ and whose 
mapping spaces are $\sLie$-algebras \eqref{Map-A-B}.  

We start with defining a degree $0$ linear map
\begin{equation} 
\label{comp_map}
\begin{split}
\cU' : \und{S}\big( \Hom(\cC(A_2), A_3) \oplus  \Hom(\cC(A_1), A_2)  \big) 
\to  \Hom(\cC(A_1), A_3)
\end{split}
\end{equation}
by using the identification 
\begin{align*}
\und{S}\bigl( \Hom(\cC(A_2), A_3) \oplus \Hom(\cC(A_1), A_2)  \bigr)
 & = \und{S}\bigl(\Hom(\cC(A_2), A_3) \bigr) \, \oplus \,
 \und{S}\bigl(\Hom(\cC(A_1), A_2) \bigr )
\\  & \quad \oplus \,
\Big( \und{S}\bigl( \Hom(\cC(A_2), A_3) \bigr) \tensor \und{S}\bigl( \Hom(\cC(A_1), A_2) \bigr) \Big)
\end{align*}
and the formulas: 
\begin{equation}
\label{cU-pr-nonzero}
\cU' \big( g \otimes (f_1 \dots f_n) \big) (X) = 
g \big( 1 \otimes f_{1} \otimes \dots \otimes f_{n} \, (\D_n(X)) \big)\,,   
\end{equation}
for all $X \in \cC(A_1)$, $g \in \Hom(\cC(A_2), A_3)$, 
$f_1, \dots, f_n \in \Hom(\cC(A_1), A_2)$,
and
\begin{equation} 
\label{cU-pr-zero}
\begin{array}{c}
\cU' \big|_{ \und{S}(\Hom(\cC(A_2), A_3)) } ~ = ~ 
\cU' \big|_{ \und{S}(\Hom(\cC(A_1), A_2)) }  ~ = ~ 0\,, \\[0.5cm]
\cU' \big|_{ \und{S}^{m \neq 1} ( \Hom(\cC(A_2), A_3) ) \otimes
\und{S} (\Hom(\cC(A_1), A_2) ) }   ~ = ~ 0\,.
\end{array}
\end{equation}

We claim that 
\begin{prop} 
\label{prop:comp}
The vector $\cU'$ defined by the above formulas 
is a MC element of the $\sLie$-algebra 
\begin{equation}
\label{Conv-for-composition}
\Hom \Big(\und{S} \big( \Hom(\cC(A_2), A_3) \oplus \Hom(\cC(A_1), A_2) 
\big) \,,\,
 \Hom(\cC(A_1), A_3) \Big)\,.
\end{equation}
\end{prop}
\begin{proof}
Let us denote by $L_{ij}$ the $\sLie$-algebra $\Hom(\cC(A_i), A_j)$
and by $d_{ij}$ the differential on $L_{ij}$. We also denote by 
$Q_{ij}$ the degree $1$ coderivation on the $\coCom$-coalgebra 
\begin{equation}
\label{und-S-Lij}
\und{S} \big( L_{ij} \big) 
\end{equation}
corresponding to the  $\sLie$-algebra $L_{ij}$\,. 
By abuse of notation, $d_{ij}$ also denotes the differential on 
\eqref{und-S-Lij} coming from the one on $L_{ij}$. 

In terms of this notation, our goal is to prove that $\cU'$ satisfies 
the equation\footnote{We sometimes use the subscript $m$ in $\{~,~,\dots, ~\}_m$ 
to denote the number of entries of the corresponding multi-bracket.}   
\begin{equation}
\label{MC-U-pr}
d_{13} \circ \cU' - \cU' \circ \big( (d_{23} + Q_{23}) \otimes 1 + 1 \otimes (d_{12} +Q_{12})  \big) +
\sum_{m=2}^{\infty} \frac{1}{m!} \{\cU', \cU', \dots, \cU' \}_m = 0\,.
\end{equation}

We will present the most bulky part of the proof of \eqref{MC-U-pr}.
Namely, we will show in detail that the sum
\begin{equation}
\label{sum-UUU-brack}
\sum_{m=2}^{\infty} \frac{1}{m!} \{\cU', \cU', \dots, \cU' \}_m 
\end{equation}
cancels with the term $-\cU' \circ (Q_{23} \otimes 1)$ in \eqref{MC-U-pr}. 
The remaining cancellations are much more straightforward and 
we leave them to the reader.

In the computations given below, we do not specify explicitly sign factors
coming from the Koszul sign rule. We address this 
issue by a short comment at the end of the proof.  

Let $g_1, \dots, g_k \in L_{23}$ and $f_1, \dots, f_n \in L_{12}$. 
Due to \eqref{cU-pr-zero}, 
$$
\frac{1}{m!} \{\cU', \cU', \dots, \cU' \}_m (g_1, \dots, g_k, f_1, \dots, f_n) = 0
$$
if $k \neq m$ or $n < k$\,. 

Unfolding the term 
\begin{equation}
\label{UUU-brack}
\frac{1}{m!} \{\cU', \cU', \dots, \cU' \}_m (g_1, \dots, g_m, f_1, \dots, f_n)
\end{equation}
with $n \ge m$, we get 
\begin{multline}
\label{UUU-brack-unfold}
\frac{1}{m!} \{\cU', \cU', \dots, \cU' \}_m (g_1, \dots, g_m, f_1, \dots, f_n) = \\
\sum_{\substack{k_1 + \dots + k_m = n \\[0.1cm] k_j \ge 1}}~
\sum_{\tau \in S_m} 
\sum_{\si \in \Sh_{k_1, k_2, \dots, k_m}} \frac{\pm 1}{m!}
\big\{
\cU'(g_{\tau(1)}, f_{\si(1)}, \dots, f_{\si(k_1)}), 
\cU'(g_{\tau(2)}, f_{\si(k_1+1)}, \dots, f_{\si(k_1+k_2)}),
\dots  \\[0.1cm]
\dots \cU'(g_{\tau(m)}, f_{\si(n-k_m+1)}, \dots, f_{\si(n)})  \big\}^{L_{13}}_m\,,
\end{multline}
where $\{~, ~, \dots, ~\}^{L_{13}}_m$ denotes the corresponding multi-bracket
on $L_{13}$ and the sign factors in the right hand side are determined by the Koszul rule.

Since  $\{~, ~, \dots, ~\}^{L_{13}}_m$ is (graded) symmetric in its argument, we can simplify 
term \eqref{UUU-brack} further: 
\begin{multline}
\label{UUU-brack-better}
\frac{1}{m!} \{\cU', \cU', \dots, \cU' \}_m (g_1, \dots, g_m, f_1, \dots, f_n) = \\
\sum_{\substack{k_1 + \dots + k_m = n \\[0.1cm] k_j \ge 1}}~
\sum_{\si \in \Sh_{k_1, k_2, \dots, k_m}} \pm
\big\{
\cU'(g_{1}, f_{\si(1)}, \dots, f_{\si(k_1)}), 
\cU'(g_{2}, f_{\si(k_1+1)}, \dots, f_{\si(k_1+k_2)}),
\dots  \\[0.1cm]
\dots \cU'(g_{m}, f_{\si(n-k_m+1)}, \dots, f_{\si(n)})  \big\}^{L_{13}}_m\,.
\end{multline}

Let $X \in \cC(A_1)$ and 
\begin{equation}
\label{Delta-X}
\D_m (X) = \sum_{\al} (\ga_{\al}; X_{\al,1}, X_{\al, 2}, \dots, X_{\al, m})  \in \big( \cC(m) \otimes \cC(A_1)^{\otimes\, m} \big)^{S_m}\,.  
\end{equation}

Then, applying equation \eqref{UUU-brack-better} and using the definition of the 
multi-bracket on $L_{13} = \Hom( \cC(A_1), A_3)$, we get 
\begin{multline}
\label{UUU-brack-applied}
\frac{1}{m!} \{\cU', \cU', \dots, \cU' \}_m (g_1, \dots, g_m, f_1, \dots, f_n) (X) = \\[0.1cm]
\sum_{\al}
\sum_{\substack{k_1 + \dots + k_m = n \\[0.1cm] k_j \ge 1}}~
\sum_{\si \in \Sh_{k_1, k_2, \dots, k_m}} 
\pm 
Q_{A_3} 
\big( \ga_{\al} ; 
\cU'(g_{1}, f_{\si(1)}, \dots, f_{\si(k_1)}) (X_{\al,1}),   \\[0.1cm]
\cU'(g_{2}, f_{\si(k_1+1)}, \dots, f_{\si(k_1 + k_2)}) (X_{\al, 2}), 
\dots,  
\cU'(g_{m}, f_{\si(n-k_m+1)}, \dots, f_{\si(n)}) (X_{\al, m}) \big)\,,
\end{multline}
where $Q_{A_3}$ denotes the coderivation of $\cC(A_3)$ corresponding 
to the $\Cobar(\cC)$-algebra structure on $A_3$. 

Using \eqref{cU-pr-nonzero}, we deduce that 
\begin{multline}
\label{UUU-brack1}
\frac{1}{m!} \{\cU', \cU', \dots, \cU' \}_m (g_1, \dots, g_m, f_1, \dots, f_n) (X) = \\[0.1cm]
\sum_{\substack{k_1 + \dots + k_m = n \\[0.1cm] k_j \ge 1}}~
\sum_{\si \in \Sh_{k_1, k_2, \dots, k_m}} 
\pm Q_{A_3} \circ
(1 \otimes g_{1}\otimes f_{\si(1)} \otimes \dots \otimes f_{\si(k_1)}
\otimes \dots \otimes g_{m} \otimes  
f_{\si(n-k_m+1)} \otimes \dots \otimes f_{\si(n)})  \\[0.1cm] 
\circ (1 \otimes \D_{k_1} \otimes \D_{k_2} \otimes \dots \otimes \D_{k_m} ) \circ \D_m (X)\,.
\end{multline}

By the axioms of the $\cC$-coalgebra structure on $\cC(A_1)$, we have 
\begin{equation}
\label{Delta-Delta-X}
(1 \otimes \D_{k_1} \otimes \D_{k_2} \otimes \dots \otimes \D_{k_m} ) \circ \D_m (X) = 
\mb \circ (\D_{\bt^{\pf}_{k_1, \dots, k_m}} \otimes 1^{\otimes\, n} )  \circ  \D_n(X)\,, 
\end{equation}
where $n = k_1 + k_2 + \dots + k_m$, 
$\D_{\bt^{\pf}_{k_1, \dots, k_m}}$ is the cooperadic comultiplication 
$$
\D_{\bt^{\pf}_{k_1, \dots, k_m}} : \cC(n) \to \cC(m) \otimes \cC(k_1) \otimes 
\cC(k_2) \otimes \dots \otimes \cC(k_m)
$$
corresponding to the planar tree $\bt^{\pf}_{k_1, \dots, k_m}$ shown on 
figure \ref{fig:fork-kkk}, and $\mb$ is the braiding isomorphism which 
``changes the positions'' of tensor factors appropriately.  
\begin{figure}[htp]
\centering 
\begin{tikzpicture}[scale=0.5]
\tikzstyle{w}=[circle, draw, minimum size=3, inner sep=1]
\tikzstyle{b}=[circle, draw, fill, minimum size=3, inner sep=1]
\node[b] (l1) at (0, 8) {};
\draw (0,8.6) node[anchor=center] {{\small $1$}};
\draw (1.5,8) node[anchor=center] {{\small $\dots$}};
\node[b] (lk1) at (3, 8) {};
\draw (3,8.6) node[anchor=center] {{\small $k_1$}};
\node[w] (n2) at (1.5, 6) {};
\node[b] (lk11) at (6, 8) {};
\draw (6,8.6) node[anchor=center] {{\small $k_1+1$}};
\draw (7.5,8) node[anchor=center] {{\small $\dots$}};
\node[b] (lk1k2) at (9, 8) {};
\draw (9.2,8.6) node[anchor=center] {{\small $k_1+k_2$}};
\node[w] (n3) at (7.5, 6) {};
\draw (11.3,7) node[anchor=center] {{\small $\dots$}};
\node[b] (lnkm1) at (14, 8) {};
\draw (13.6,8.6) node[anchor=center] {{\small $n-k_m+1$}};
\draw (15.5,8) node[anchor=center] {{\small $\dots$}};
\node[b] (ln) at (17, 8) {};
\draw (17,8.6) node[anchor=center] {{\small $n$}};
\node[w] (nm1) at (15.5, 6) {};
\node[w] (n1) at (8.5, 3) {};
\node[b] (r) at (8.5, 2) {};
\draw (n2) edge (l1);
\draw (n2) edge (lk1);
\draw (n3) edge (lk11);
\draw (n3) edge (lk1k2);
\draw (nm1) edge (lnkm1);
\draw (nm1) edge (ln);
\draw (n1) edge (n2);
\draw (n1) edge (n3);
\draw (n1) edge (nm1);
\draw (r) edge (n1);
\end{tikzpicture}
\caption{The labeled planar tree $\bt^{\pf}_{k_1, \dots, k_m}$} 
\label{fig:fork-kkk}
\end{figure}

Unfolding $\D_n (X)$ 
$$
\D_n(X) =  \sum_{\beta} (\ti{\ga}_{\beta}; \ti{X}_{\beta, 1}, \dots, \ti{X}_{\beta, n})
$$ 
and using the fact that $\D_n$ lands in the space of $S_n$-invariants, we rewrite 
the expression $(\D_{\bt^{\pf}_{k_1, \dots, k_m}} \otimes 1^{\otimes\, n} )  \circ  \D_n(X)$
as follows: 
\begin{multline}
\label{compute-Delta-Delta}
(\D_{\bt^{\pf}_{k_1, \dots, k_m}} \otimes 1^{\otimes\, n} )  \circ  \D_n(X) =  \\[0.1cm]
\sum_{\beta} \pm
(\D_{\bt^{\pf}_{k_1, \dots, k_m} } \otimes 1^{\otimes\, n} ) 
( \si^{-1} (\ti{\ga}_{\beta}); \ti{X}_{\beta, \si(1)}, \dots, \ti{X}_{\beta, \si(n)}) 
=   \\[0.1cm]
\sum_{\beta} \pm
(\D_{\si(\bt^{\pf}_{k_1, \dots, k_m}) } \otimes 1^{\otimes\, n} ) 
( \ti{\ga}_{\beta}; \ti{X}_{\beta, \si(1)}, \dots, \ti{X}_{\beta, \si(n)}) 
\end{multline}
for any $\si \in S_n$\,.

Combining this observation with \eqref{Delta-Delta-X}, we conclude that 
\begin{multline}
\label{UUU-brack11}
\frac{1}{m!} \{\cU', \cU', \dots, \cU' \}_m (g_1, \dots, g_m, f_1, \dots, f_n) (X)  \\[0.1cm]
= \sum_{\beta} 
\sum_{\substack{k_1 + \dots + k_m = n \\[0.1cm] k_j \ge 1}}~
\sum_{\si \in \Sh_{k_1, k_2, \dots, k_m}} 
\pm Q_{A_3} \circ (1 \otimes g_1 \otimes \dots \otimes g_m)  \\[0.1cm] \circ 
(\D_{\si(\bt^{\pf}_{k_1, \dots, k_m}) } \otimes 1^{\otimes\, n} ) 
( \ti{\ga}_{\beta}; f_{\si(1)}(\ti{X}_{\beta, \si(1)}), \dots, f_{\si(n)}(\ti{X}_{\beta, \si(n)})) \,.
\end{multline}

On the other hand, 
\begin{multline}
\label{line-Delta-Delta}
\sum_{\beta} 
\sum_{\substack{k_1 + \dots + k_m = n \\[0.1cm] k_j \ge 1}}~
\sum_{\si \in \Sh_{k_1, k_2, \dots, k_m}} 
\pm (\D_{\si(\bt^{\pf}_{k_1, \dots, k_m}) } \otimes 1^{\otimes\, n} ) 
( \ti{\ga}_{\beta}; f_{\si(1)}(\ti{X}_{\beta, \si(1)}), \dots, f_{\si(n)}(\ti{X}_{\beta, \si(n)}))  \\[0.1cm]
= \D_m \Big(  \sum_{\beta} \pm (\ti{\ga}_{\beta}; f_{1}(\ti{X}_{\beta, 1}), \dots, f_n(\ti{X}_{\beta, n}))  \Big)\,.
\end{multline}
 
Therefore, by definition of the $\sLie$-structure on $L_{23} = \Hom(\cC(A_2), A_3)$, we have  
\begin{multline}
\label{UUU-brack-OK}
\frac{1}{m!} \{\cU', \cU', \dots, \cU' \}_m (g_1, \dots, g_m, f_1, \dots, f_n) (X) = \\[0.1cm]
 \cU'  (\{g_1, \dots, g_m \}^{L_{23}}_m, f_1, \dots, f_n) \, (X)\,.
\end{multline}

Let us also observe that, due to  \eqref{cU-pr-zero},   
$$
\big( \cU' \circ (Q_{23} \otimes 1) (g_1, \dots, g_m, f_1, \dots, f_n) \big) (X) = 
\cU'  (\{g_1, \dots, g_m \}^{L_{23}}_m, f_1, \dots, f_n) (X)\,.
$$

Thus, 
\begin{multline}
\label{UUU-brack-OK1}
\frac{1}{m!} \{\cU', \cU', \dots, \cU' \}_m (g_1, \dots, g_m, f_1, \dots, f_n) (X) = \\[0.1cm]
\big( \cU' \circ  (Q_{23} \otimes 1)  (g_1, \dots, g_m, f_1, \dots, f_n) \big) (X)
\end{multline}
which implies the desired cancellation of sum 
\eqref{sum-UUU-brack} with the term $- \cU' \circ (Q_{23} \otimes 1)$ 
in \eqref{MC-U-pr}.  

Let us now address the issue of sign factors.  
The sign factor in front of the term 
\begin{equation}
\label{term-UUU-brack}
Q_{A_3} \circ (1 \otimes g_1 \otimes \dots \otimes g_m)  \\[0.1cm] \circ 
(\D_{\si(\bt^{\pf}_{k_1, \dots, k_m}) } \otimes 1^{\otimes\, n} ) 
( \ti{\ga}_{\beta}; f_{\si(1)}(\ti{X}_{\beta, \si(1)}), \dots, f_{\si(n)}(\ti{X}_{\beta, \si(n)}))
\end{equation}
in \eqref{UUU-brack11} comes from rearranging the homogeneous vectors
\begin{equation}
\label{stand-order}
f_1, f_2, \dots, f_n,  \ti{\ga}_{\beta}, \ti{X}_{\beta, 1}, \ti{X}_{\beta, 2}, 
\dots \ti{X}_{\beta, n}
\end{equation}
from their standard order in \eqref{stand-order} to the order in which they appear 
in \eqref{term-UUU-brack}. It is easy to see that we get the same sign factors
in front of the corresponding terms, when we unfold the right hand 
side of \eqref{UUU-brack-OK}. 

Proposition \ref{prop:comp} is proved. 
\end{proof}

Combining Lemma \ref{lem:infty-morph} with Proposition \ref{prop:comp}
we deduce that 
\begin{cor}
\label{cor:cU-intro}
The map $\cU'$ defined by equations \eqref{cU-pr-nonzero} 
and \eqref{cU-pr-zero} lifts to an $\infty$-morphism 
\begin{equation}
\label{cU-comp}
\cU \maps \map(A_2, A_3) \oplus \map(A_1, A_2) \to \map(A_1, A_3)\,.
\end{equation}
\qed
\end{cor}
\begin{remark}
\label{rem:U-pr-again}
We would like to remark that the map 
$$
\cU'  \maps \und{S} \big( \Hom(\cC(A_2), A_3) \oplus  \Hom(\cC(A_1), A_2)  \big) 
\to  \Hom(\cC(A_1), A_3),
$$
can be equivalently defined by the single formula: 
\begin{equation} 
\label{U-pr-again}
\cU' \bigl( (g_1 \oplus f_1),(g_2  \oplus f_2),\ldots,(g_n  \oplus f_n) \bigr) = 
\sum_{i=1}^{n} \pm g_{i} \circ  \bigl(1 \tensor f_1 \tensor \cdots \tensor
\widehat{f_{i}} \tensor \cdots \tensor f_{n} \bigr) \circ \D_{n-1}
\end{equation}
where $(g_i \oplus f_i) \in \Hom(\cC(A_2), A_3)
\oplus  \Hom(\cC(A_1), A_2)$ and $\pm$ is the usual Koszul sign factor.  
\end{remark}

The following theorem shows that the composition in 
$\HoAlg_{\cC}$ given by $\infty$-morphism \eqref{cU-comp}
is associative. 
\begin{thm} 
\label{thm:comp-assoc}
Let $A_1, \dots, A_4$ be $\Cobar(\cC)$-algebras and let 
$$
\cU_{i_1 i_2 i_3} \maps \map(A_{i_2} , A_{i_3}) \oplus \map(A_{i_1}, A_{i_2}) \to \map(A_{i_1}, A_{i_3})
$$
be the $\infty$-morphism given in Corollary \ref{cor:cU-intro}. Then the 
following diagram 
\begin{equation} 
\label{pent}
\begin{tikzpicture}[ext/.style={rectangle, minimum size=4, inner sep=1}]
\node[ext] (123-4) at (0,1.5) {$\map(A_3, A_4) \oplus \big( \map(A_2,A_3) \oplus \map(A_1,A_2) \big)$};
\node[ext] (1-234) at (0,0) {$\big( \map(A_3, A_4) \oplus  \map(A_2,A_3) \big) \oplus \map(A_1,A_2)$};
\node[ext] (134) at (6, 3) {$\map(A_3, A_4) \oplus \map(A_1,A_3)$};
\node[ext] (124) at (6, -1.5) {$\map(A_2, A_4) \oplus \map(A_1,A_2)$};
\node[ext] (14) at (10, 0.75) {$\map(A_1, A_4) $};
\path[->,font=\scriptsize]
(0,1.1) edge  node[auto] {$\cong$} (0, 0.4)
(1,2)  edge  node[auto] {$~~\id \otimes \cU_{123}$} (3,2.8)
(1,-0.5) edge  node[below] {$ \cU_{234} \otimes \id~~$} (3,-1.3) 
(8,2.5) edge  node[auto] {$ \cU_{134}$} (9.5,1.2) 
(8,-1) edge  node[below] {$~~~~\cU_{124}$} (9.5,0.3) ;
\end{tikzpicture}
\end{equation}
commutes.
\end{thm}
\begin{proof}
Let  $h \in \map(A_3,A_4)$,   $g_{1}, \ldots, g_{m} \in \map(A_2,A_3)$, 
and $f_{1}, \ldots, f_{n} \in \map(A_1,A_2)$\,. Composing the 
lower arrows in \eqref{pent} with the canonical projection 
\begin{equation}
\label{p-14}
p_{14} :  \und{S}(\map(A_1,A_4)) \to  \map(A_1, A_4), 
\end{equation}
we get 
\begin{equation} 
\label{enrich_eq1}
\begin{split}
&p_{14}\circ  \cU_{124} \circ  (\cU_{234} \tensor \id )(h, g_1, \dots, g_m, 
f_1, \dots, f_n ) =\\
& \quad  h
\Bigl( (1 \tensor g_1 \tensor \cdots \tensor g_{m})\circ  \Delta_{m} \circ
 (1 \tensor f_{1} \tensor \cdots \tensor f_{n}) \circ \Delta_{n} \Bigr)\,.
\end{split}
\end{equation}

Similarly, composing the upper arrows in \eqref{pent}
with canonical projection \eqref{p-14}, we get   
\begin{equation} 
\label{enrich_eq2}
\begin{split}
& p_{14} \circ \cU_{134} \circ (\id \tensor \, \cU_{123})(h, g_1, \dots, g_m, f_1, \dots, f_n) = \\
&\sum_{\substack{k_1+\cdots + k_{m} =n \\ \sigma \in
    \Sh(k_1,\ldots,k_m)}} \pm
h \Bigl ( 1 \tensor g_1(1 \tensor f_{\sigma(1)} \tensor
f_{\sigma(2)} \tensor   \cdots \tensor f_{\sigma(k_1)}) 
 \tensor  g_2(1 \tensor f_{\sigma(k_1 +1)} \tensor \cdots \tensor f_{\sigma(k_1+
  k_2)}) \tensor
 \cdots  \\
&\quad \quad \tensor g_m(1 \tensor f_{\sigma(n-k_m +1)} \tensor \cdots
 \tensor f_{\sigma(n)})  \circ (1 \tensor \Delta_{k_1} \tensor
 \Delta_{k_2} \tensor  \cdots \tensor \Delta_{k_m} ) \Bigr) \Delta_m,
\end{split}
\end{equation}
where the sign factors are determined by the Koszul rule. 

Using equation \eqref{Delta-Delta-X}, computation \eqref{compute-Delta-Delta}, and 
equation \eqref{line-Delta-Delta} from the proof of Proposition \ref{prop:comp},
we conclude that the left hand side of 
\eqref{enrich_eq1} coincides with the left hand side of \eqref{enrich_eq2}. 
Since any $\infty$-morphism to $\map(A_1, A_4)$ is uniquely determined by 
its composition with projection \eqref{p-14}, we deduce that 
$$
 \cU_{124} \circ  (\cU_{234} \tensor \id ) =  \cU_{134} \circ (\id \tensor\, \cU_{123})\,.
$$

Thus diagram \eqref{pent} is indeed commutative. 
\end{proof}

\subsection{The proof of the unit axiom}

Let us recall that $\bfzero$ is the unit object in the category $\tLie$ and 
observe that for every $\Cobar(\cC)$-algebra $A$ we have a canonical 
enhanced morphism 
\begin{equation}
\label{unit-morphism}
\bfzero ~\xto{(\id_A, \, 0)}~ \map(A,A)\,,
\end{equation}
where $\id_A$ is the MC element of $\map(A,A) = \Hom(\cC(A), A)$
corresponding to the identity $\infty$-morphism
$$
\id_{\cC(A)} : \cC(A) \to  \cC(A)
$$ 
and $0$ is the unique $\sLie$-morphism from $\bfzero $ to $\map(A,A)$.

We claim that 
\begin{prop}
\label{prop:unit}
For every pair $A, B$ of homotopy algebras of type $\cC$, 
the diagrams 
\begin{equation}
\label{unit-right}
\begin{tikzpicture}
\matrix (m) [matrix of math nodes, row sep=3em, column sep=3em]
{\map(A,B) \oplus \map(A,A) & \map(A,B)   ~ \\
 \map (A,B) \oplus \bfzero & ~ \\ };
\path[->, font=\scriptsize]
(m-1-1) edge node[above] {$\cU$} (m-1-2) 
(m-2-1) edge  node[auto] {$ \id_{\map(A,B)} \otimes (\id_A, \, 0)$} (m-1-1)
edge (m-1-2);
\end{tikzpicture}
\end{equation}
\begin{equation}
\label{unit-left}
\begin{tikzpicture}
\matrix (m) [matrix of math nodes, row sep=3em, column sep=3em]
{\map(B,B) \oplus \map(A,B) & \map(A,B)   ~ \\
  \bfzero \oplus \map(A,B) & ~ \\ };
\path[->, font=\scriptsize]
(m-1-1) edge node[above] {$\cU$} (m-1-2) 
(m-2-1) edge  node[auto] {$(\id_B, \, 0) \otimes  \id_{\map(A,B)}$} (m-1-1)
edge (m-1-2); 
\end{tikzpicture}
\end{equation}
commute.
\end{prop}
\begin{proof}
Let us denote by 
$$
\cK ~:~  \map (A,B) \oplus \bfzero ~\to~  \map (A,B)
$$
the composition of the vertical arrow and the horizontal arrow in  \eqref{unit-right}
and let $\cK'$ be the corresponding element in 
$$
\Hom \big( \und{S}( \map (A,B) \oplus \bfzero),\, \map(A,B) \big)\,.
$$ 

Unfolding the definition of  composition of enhanced morphisms in $\tLie$ (see Proposition 3.4 
in \cite{EnhancedLie}), we get that 
\begin{equation}
\label{cK-pr}
\cK'(g_1, \dots, g_m) = \sum_{n \ge 0} \frac{1}{n!} 
 \cU'\big( (g_1, \dots, g_m) \otimes \id^{n}_A \big)\,,
\end{equation}
where $g_1, \dots, g_m \in \map (A,B)$.

Hence, using \eqref{cU-pr-nonzero} and \eqref{cU-pr-zero} we deduce that 
$$
\cK'(g_1, \dots, g_m) = 0 \qquad \forall ~~ m \ge 2
$$
and, for every $X \in \cC(A)$,
\begin{equation}
\label{cK-pr-unfold}
\cK'(g_1) (X) =  \sum_{n \ge 0} \frac{1}{n!} \,
g_1 \big( (1\otimes  \id^{n}_A)   \D_n(X) \big) =  g_1(X)\,,
\end{equation}
where the last equality follows from the identification 
of $\cC(\cC(A) )^{\inv}$ and $\cC(\cC(A))$ via the inverse 
of isomorphism \eqref{coinvar-invar}. 

Thus diagram \eqref{unit-right} indeed commutes. 

The proof of the commutativity of \eqref{unit-left} is easier.
So we leave it to the reader. 
\end{proof}

\subsection{The simplicial category $\HoAlg_{\cC}^{\sD}$ of homotopy algebras}

Let $\Omega_{n}=\Omega^{\bullet}(\Delta^{n})$ denote the polynomial de
Rham complex on the $n$-simplex with coefficients in $\bbk$, and
$\{\Omega_{n} \}_{n \geq 0}$ the associated simplicial
dg commutative $\bbk$-algebra. Let us recall \cite[Proposition 4.1]{EnhancedLie} that 
for every filtered $\sLie$-algebra $L$ the simplicial set $\mMC_{\bul}(L)$ 
with\footnote{Recall that $\MC(L)$ denotes the set of MC elements 
of a filtered $\sLie$-algebras $L$ \cite{EnhancedLie}.} 
$$
\mMC_n(L) : = \MC(L \hotimes \Om_n )
$$
is a Kan complex (a.k.a. an $\infty$-groupoid). We call the simplicial set 
$\mMC_{\bul}(L)$ the {\it Deligne-Getzler-Hinich (DGH) $\infty$-groupoid}. 

Let us also recall (see Theorem 4.2 in \cite{EnhancedLie}) that applying 
the functor $\mMC_{\bul}$ to mapping spaces of any $\tLie$-enriched category, 
we get a simplicial category whose mapping spaces are Kan complexes.  
Thus, applying \cite[Theorem 4.2]{EnhancedLie} to the $\tLie$-enriched category  
$\HoAlg_{\cC}$ and using Lemma \ref{lem:infty-morph} we deduce the 
following theorem:
\begin{thm}
\label{thm:HoAlg-Delta}
Let $\cC$ be a coaugmented dg cooperad satisfying 
Conditions \eqref{cond:cC-circ-filtered} and \eqref{cC-reduced}. 
Then the assignment  
$$
(A, B) ~\in~ \Objects(\Cat_{\cC})  \times  \Objects(\Cat_{\cC}) ~~ \mapsto ~~ 
\mMC_{\bul}(\map(A,B)) 
$$
gives us a category enriched over $\infty$-groupoids (a.k.a. Kan complexes).
Moreover, for every pair of $\Cobar(\cC)$-algebras $A, B$, the set $\mMC_{0}(\map(A,B))$
is in bijection with the set of $\infty$-morphisms from $A$ to $B$\,. \qed 
\end{thm}
 
\section{$\pi_0 \big(\HoAlg^{\sD}_{\cC} \big)$ is a correct homotopy category of homotopy algebras}
\label{sec:HT}

Let $A$ and $B$ be homotopy algebras of type $\cC$ and
$F$ be an $\infty$-morphism from $A$ to $B$: 
$$
F ~:~ \Big( \cC(A), \pa + Q_A  \Big) ~\to~ \Big( \cC(B), \pa + Q_B  \Big)\,.
$$ 

Composing $F$ with a canonical projection $p_B :  \cC(B) \to B$, 
and restricting this composition to $A \subset \cC(A)$, we get 
a map of cochain complexes:
\begin{equation}
\label{F-lin-term}
p_B  \circ F \Big|_{A} : A \to B\,.
\end{equation}
We refer to \eqref{F-lin-term} as {\it the linear term} of the $\infty$-morphism $F$.  
Recall that an $\infty$-morphism $F$ is called an $\infty$ {\it quasi-isomorphism} if 
its linear term \eqref{F-lin-term} is a quasi-isomorphism of cochain complexes. 

Let us recall that $\Cat_{\cC}$ is the category of $\Cobar(\cC)$-algebras with 
morphisms being $\infty$-morphisms, and observe that we have the obvious functor 
\begin{equation}
\label{mF}
\mF : \Cat_{\cC} \to \pi_0 \big(\HoAlg^{\sD}_{\cC} \big) 
\end{equation}
which acts by identity on objects and assigns to every $\infty$-morphism 
the isomorphism class of the corresponding MC element. 

The goal of this section is to prove that the category $\pi_0 \big(\HoAlg^{\sD}_{\cC} \big)$ is 
the homotopy category for $\Cat_{\cC}$. Namely, 
\begin{thm}
\label{thm:pi-0-HoAlg}
The functor $\mF$ sends $\infty$ quasi-isomorphisms to isomorphisms
and it is a universal functor with this property. I.e., if  $\mG : \Cat_{\cC} \to \mD$ is a functor 
which sends $\infty$ quasi-isomorphisms to isomorphisms in $\mD$ then 
there exists a unique functor  
$$
\mG' : \pi_0\big( \HoAlg^{\sD}_{\cC} \big) \to \mD
$$
such that $\mG = \mG' \circ \mF$\,.
\end{thm}

\subsection{``Inverting'' $\infty$ quasi-isomorphisms}
Let $A_1, A_2, A_3$ be $\Cobar(\cC)$-algebras, $F$ be an 
$\infty$-morphism from $A_1$ to $A_2$, and $F'$ be the MC element of $\map(A_1, A_2)$
corresponding to $F$. Let us denote by $F' \oplus 0$, $0 \oplus F'$ the corresponding 
MC elements of the $\sLie$-algebras
$$
\map(A_1, A_2) \oplus \map(A_3 , A_1) 
$$ 
and 
$$
\map(A_2, A_3) \oplus \map(A_1, A_2)\,,
$$
respectively. 

Twisting the $\sLie$-morphisms 
$$
\cU : \map(A_1, A_2) \oplus \map(A_3 , A_1)  \to \map(A_3, A_2) 
$$
and 
$$
\cU : \map(A_2, A_3) \oplus \map(A_1, A_2)  \to \map(A_1, A_3) 
$$
by the MC elements  $F' \oplus 0$, $0 \oplus F'$, respectively, and 
composing the resulting  $\sLie$-morphisms with the canonical maps 
$$
f \mapsto 0 \oplus f ~: ~ 
\map(A_3, A_1)  \to   \map(A_1, A_2) \oplus \map(A_3 , A_1)
$$
$$
f \mapsto f \oplus 0 ~: ~ 
 \map(A_2, A_3)  \to  \map(A_2, A_3) \oplus \map(A_1, A_2)
$$ 
we get two $\sLie$-morphisms 
\begin{equation}
\label{comp-A-312}
\cU_{A_3 A_1 A_2} :  \map(A_3, A_1) \to  \map(A_3, A_2) 
\end{equation}
and 
\begin{equation}
\label{comp-A-123}
\cU_{A_1 A_2  A_3} :  \map(A_2, A_3) \to \map(A_1, A_3)\,.
\end{equation}

The following proposition says that the induced maps of MC elements 
\begin{equation}
\label{comp-F-312}
(\cU_{A_3 A_1 A_2})_* : \MC( \map(A_3, A_1) ) \to \MC(  \map(A_3, A_2)  )
\end{equation}
and 
\begin{equation}
\label{comp-F-123}
(\cU_{A_1 A_2 A_3})_* : \MC( \map(A_2, A_3) ) \to \MC( \map(A_1, A_3) )
\end{equation}
correspond to the composition (resp. the pre-composition) of 
an $\infty$-morphism from $A_3$ to $A_1$ (resp. from $A_2$ to $A_3$)
with $F$: 
\begin{prop}
\label{prop:cU-star-comp}
If $G$ is an $\infty$-morphism from $A_3$ to $A_1$ and $G'$ is the 
corresponding MC element of $\map(A_3, A_1)$ then the MC element 
$(\cU_{A_3 A_1 A_2})_* (G')$ of $\map(A_3, A_2)$ corresponds to 
the composition $F \circ G$. Similarly, if $G$ is an $\infty$-morphism from $A_2$ to 
$A_3$ and $G'$ is the 
corresponding MC element of $\map(A_2, A_3)$ then the MC element 
$(\cU_{A_1 A_2 A_3})_* (G')$ of $ \map(A_1, A_3) $ corresponds to 
the composition $G \circ F$\,.
\end{prop}
\begin{proof} The proof of these statements is straightforward. \\
\end{proof}

We also claim that
\begin{prop}
\label{prop:cU-F-q-iso}
$\sLie$-morphisms \eqref{comp-A-312} and \eqref{comp-A-123}
are compatible with the filtrations $\cF^{\ari}_{\bul}$ from \eqref{Conv-filter}.
Furthermore,  if $F$ is an $\infty$ quasi-isomorphism, 
then \eqref{comp-A-312} and \eqref{comp-A-123} give us  
$\sLie$ quasi-isomorphisms
\begin{equation}
\label{q-iso-cF-m}
\cF^{\ari}_{m}  \map(A_3, A_1) \to \cF^{\ari}_{m} \map(A_3, A_2) 
\quad \textrm{and} \quad 
\cF^{\ari}_{m} \map(A_2, A_3) \to \cF^{\ari}_{m} \map(A_1, A_3)
\end{equation}
respectively, for every $m \ge 1$\,.
\end{prop} 
\begin{proof}
Due to Remark \ref{rem:U-pr-again}, the composition 
$$
\cU'_{A_3 A_1 A_2} : = p_{\map(A_3, A_2)} \circ  \cU_{A_3 A_1 A_2} :  \und{S}\big(\map(A_3, A_1)\big) \to \map(A_3, A_2)
$$
is given by the formula 
\begin{equation}
\label{cU-A-312-pr}
\cU'_{A_3 A_1 A_2} (g_1, \dots, g_n) (X) =  
F' \big( 1 \otimes g_{1} \otimes \dots \otimes g_{n} \, (\D_n(X)) \big)\,,  
\end{equation}
where $X \in \cC(A_3)$ and $g_1, \dots, g_n \in \map(A_3, A_1)$\,. 
Similarly, the composition 
$$
\cU'_{A_1 A_2 A_3} : = p_{\map(A_1, A_3)} \circ \cU_{A_1 A_2 A_3} :  
\und{S}\big( \map(A_2, A_3) \big) \to \map(A_1, A_3)
$$
is given by the formula 
 \begin{equation}
\label{cU-A-123-pr}
\cU'_{A_1 A_2 A_3} (h_1, \dots, h_n) (Y) =
\begin{cases}
\displaystyle \sum_{m \ge 1}
\frac{1}{m!} h_1 \big( 1 \otimes F' \otimes \dots \otimes F' \, (\D_m(Y)) \big)   \qquad {\rm if} ~~  n=1 \,, \\[0.5cm]
 0 \qquad {\rm otherwise}\,,
\end{cases}
\end{equation}
where $Y \in \cC(A_1)$ and $h_1, \dots, h_n \in \map(A_2, A_3)$\,. 
  
The compatibility of map $\cU'_{A_3 A_1 A_2}$ and $\cU'_{A_1 A_2 A_3}$ with the filtrations 
$\cF^{\ari}_{\bul}$ can be checked directly by unfolding the right hand side 
of \eqref{cU-A-312-pr} and the right hand side of \eqref{cU-A-123-pr}, respectively. 

We also see that the linear terms
\begin{equation}
\label{cU-1-A-312}
\cU_{1, A_3 A_1 A_2} :  \map(A_3, A_1) \to  \map(A_3, A_2) 
\end{equation}
and 
\begin{equation}
\label{cU-1-A-123}
\cU_{1, A_1 A_2 A_3} :   \map(A_2, A_3) \to \map(A_1, A_3)
\end{equation}
of the $\sLie$-morphisms $\cU_{A_3 A_1 A_2}$ and $\cU_{A_1 A_2 A_3}$ are 
given by the formulas: 
$$
\cU_{1, A_3 A_1 A_2} (g) (X) = F' \big( (1 \otimes g) \circ \D_1(X) \big)
$$
and 
$$
\cU_{1, A_1 A_2 A_3}(h) (Y) = \sum_{k \ge 1}
\frac{1}{k!} h \big( 1 \otimes F' \otimes \dots \otimes F' \, (\D_k(Y)) \big)\,, 
$$
respectively, where $g \in \map(A_3, A_1)$, $h \in \map(A_2, A_3)$  $X \in \cC(A_3)$ and $Y \in \cC(A_1)$.   

Let $\vf :  A_1 \to A_2$ be the linear term 
\begin{equation}
\label{vf}
\vf : = p_{A_2} \circ F \big|_{A_1}
\end{equation}
of the $\infty$-morphism $F$ and assume that 
$\vf$ is a quasi-isomorphism of cochain complexes. 

To prove that $\cU_{1, A_1 A_2 A_3}$ induces an isomorphism 
$$
H^{\bul} \big( \cF^{\ari}_{m} \map(A_2, A_3) \big) \to H^{\bul} \big( \cF^{\ari}_{m} \map(A_1, A_3) \big)
$$
for every $m$, we observe that, for $m \ge 2$ 
$$
\cF^{\ari}_{m} \map(W, A_3)  =   \cF^{\ari}_{m} \Hom(\cC_{\c}(W), A_3) 
$$
and  $\map(W, A_3) = \cF^{\ari}_{1} \map(W, A_3)$ splits into the direct sum of 
cochain complexes
\begin{equation}
\label{Map-W-A3-splits}
\map(W, A_3) =  \Hom(W, A_3) ~\oplus~ \Hom(\cC_{\c}(W), A_3)\,,
\end{equation}
where $W$ is either $A_2$ or $A_1$, and $\cC_{\c}$ is the cokernel of 
the coaugmentation of $\cC$.  
 
We also observe that the chain map $\cU_{1, A_1 A_2 A_3}$ is 
compatible with decomposition \eqref{Map-W-A3-splits} and 
the corresponding chain map 
$$
\Hom(A_2, A_3) \to  \Hom(A_1, A_3)
$$  
is a quasi-isomorphism because the functor $\Hom(-,-)$ preserves 
quasi-isomorphisms in $\Ch_{\bbk}$\,.

So we should now prove that the chain map 
\begin{equation}
\label{cU1-123-cC-circ}
\cU_{1, A_1 A_2 A_3} \Big|_{\cF^{\ari}_{m} \Hom(\cC_{\c}(A_2), A_3) }   ~:~ \cF^{\ari}_{m} \Hom(\cC_{\c}(A_2), A_3)
\to   \cF^{\ari}_{m} \Hom(\cC_{\c}(A_1), A_3)  
\end{equation}
induces an isomorphism on cohomology for every $m \ge 1$\,.
 
For this purpose, we equip the cochain complex $\cF^{\ari}_{m}  \Hom(\cC_{\c}(W), A_3)$
with the following descending filtration
$$
\cF^{\ari}_{m}  \Hom(\cC_{\c}(W), A_3) = \cF^{\cC}_0 \cF^{\ari}_{m}  \Hom(\cC_{\c}(W), A_3)
\supset  \cF^{\cC}_1 \cF^{\ari}_{m}  \Hom(\cC_{\c}(W), A_3)  \supset
$$
$$
\supset  \cF^{\cC}_2 \cF^{\ari}_{m}  \Hom(\cC_{\c}(W), A_3) \supset \dots
$$
\begin{equation}
\label{cF-cC-Map}
\cF^{\cC}_q \, \cF^{\ari}_{m}  \Hom(\cC_{\c}(W), A_3) ~ : = ~
\end{equation}
$$
\big\{ g \in  \cF^{\ari}_{m}   \Hom(\cC_{\c}(W), A_3) ~~ \big|~~ g (X)  = 0 ~~\forall~ X \in \cF^{q}\cC_{\c}(W) \big\}\,,
$$
where $\cF^{\bul} \cC_{\c}$ is the ascending filtration on the pseudo-cooperad
$\cC_{\c}$ from \eqref{cC-circ-filtr} and $W$ is, as above, either $A_1$ or $A_2$. 

Due to inclusion \eqref{D-bt-filtr}, the differential on $\map(W,A_3)$ is 
compatible with the filtration. Moreover, condition \eqref{cocomplete}
implies that  $\cF^{\ari}_{m}  \Hom(\cC_{\c}(W), A_3)$ is complete with respect to
this filtration, i.e.
$$
\cF^{\ari}_{m}   \Hom(\cC_{\c}(W), A_3) = \lim_q ~ \cF^{\ari}_{m}   \Hom(\cC_{\c}(W), A_3)  ~
\big/ ~\cF^{\cC}_q \cF^{\ari}_{m}  \Hom(\cC_{\c}(W), A_3)\,.
$$

Let us denote by $E_{m,q}(\cC, W)$ the following cochain complex 
\begin{equation}
\label{E-m-q}
E_{m,q}(\cC_{\c}, W) : = 
\bigoplus_{n \ge m} \Big(  \big( \cF^q \cC_{\c}(n) / \cF^{q-1} \cC_{\c}(n) \big) \,  \otimes \, W^{\otimes\, n}  \Big)_{S_n} \,. 
\end{equation}

It is clear that,  the associated graded complex 
$$
\Gr_{\cF^{\cC}} \, \cF^{\ari}_m  \Hom(\cC_{\c}(W), A_3) 
$$
is isomorphic to  
\begin{equation}
\label{assoc-Gr-m-ge2}
\bigoplus_{q \ge 1} \Hom\big(E_{m,q}(\cC_{\c}, W) ,   A_3  \big)\,.
\end{equation}
Furthermore, the differential $\pa_{\Gr}$ on \eqref{assoc-Gr-m-ge2} comes from 
those on $W$, $\cC_{\c}$ and $A_3$\,.

The map between the associated graded complexes
\begin{equation}
\label{assoc-Gr-map}
\cU^{\Gr}_{1, A_1 A_2 A_3} ~ : ~  
\Hom\big(E_{m,q}(\cC_{\c}, A_2) ,   A_3  \big) \to  \Hom\big(E_{m,q}(\cC_{\c}, A_1) ,   A_3  \big)
\end{equation}
induced by \eqref{cU1-123-cC-circ} is given by the formula
\begin{equation}
\label{assoc-Gr-map-dfn}
\cU^{\Gr}_{1, A_1 A_2 A_3} (g)(\ga; a_1, \dots, a_n)  = g (\ga; \vf(a_1), \dots, \vf(a_n))\,,
\end{equation}
where $\vf$ is the linear term of the $\infty$-morphism $F$ and
$\ga \in  \cF^q \cC_{\c}(n) / \cF^{q-1} \cC_{\c}(n)$\,.
 
Using the K\"unneth theorem, the fact that 
$\vf$ induces an isomorphism $H^{\bul}(A_1) \to H^{\bul}(A_2)$,
and $\textrm{char}(\bbk) = 0$ we conclude that the chain map 
$$
E_{m,q}(\cC_{\c}, A_1)  \to E_{m,q}(\cC_{\c}, A_2)
$$ 
induced by $\vf$ is a quasi-isomorphism.  

Therefore, since the functor $\Hom(-,-)$ preserves quasi-isomorphisms
in $\Ch_{\bbk}$, we deduce chain map \eqref{assoc-Gr-map} between the 
associated graded complexes is a quasi-isomorphism.  
  
Since  $\cF^{\ari}_{m}  \Hom(\cC_{\c}(W), A_3)$ is complete with respect to filtration 
\eqref{cF-cC-Map},  and filtration \eqref{cF-cC-Map} is bounded from the 
left, applying Lemma D.1 from \cite{DeligneTw} to the cone of chain map 
\eqref{cU1-123-cC-circ}, we conclude that  \eqref{cU1-123-cC-circ}, and hence
\eqref{cU-1-A-123}, is indeed a quasi-isomorphism 
of cochain complexes. 
 
A similar argument shows that map \eqref{cU-1-A-312} is a quasi-isomorphism 
of cochain complexes, provided so is $\vf$ \eqref{vf}. 

Proposition \ref{prop:cU-F-q-iso} is proved. 
\end{proof} 
 
Let $A$ and $B$ be $\Cobar(\cC)$-algebras.  
We will now prove that every $\infty$ quasi-morphism $F$ from $A$ to $B$  
is ``invertible'' in the following sense: 
\begin{cor}
\label{cor:invertible}
Let $F$ be an $\infty$ quasi-isomorphism from $A$ to $B$. 
Then there exists an $\infty$-morphism $G$ from $B$ to $A$
such that the MC element $(G \circ F)'$ (resp.  $(F \circ G)'$)
of the $\sLie$-algebra $\map(A,A)$ (resp. $\map(B,B)$) 
corresponding to the composition $G \circ F$ (resp. $F \circ G$)
is isomorphic in $\mMC_{\bul}(\map(A,A))$ (resp. in $\mMC_{\bul}(\map(B,B))$)
to the MC element $\id'_A$ (resp. $\id'_B$) corresponding to the 
identity morphism $\id_A$ (resp. $\id_B$). If $\wt{G}$ is another 
$\infty$-morphism from $B$ to $A$ satisfying the above properties then
the MC element $\wt{G}' \in \map(B,A)$ corresponding to $\wt{G}$ is isomorphic 
to $G'$ in $\mMC_{\bul}(\map(B,A))$\,.
\end{cor}
\begin{proof} Let us start with the question of existence of $G$.

Due to Proposition \ref{prop:cU-star-comp}, it suffices to 
prove that there exists a MC element $G'$ of $\map(B,A)$ 
such that the $0$-cell
\begin{equation}
\label{G-pr-to-A}
(\cU_{A B A})_* (G') \in \mMC_0(\map(A,A))
\end{equation}
is connected to $\id'_A$ and the $0$-cell  
\begin{equation}
\label{G-pr-to-B}
(\cU_{B A B})_* (G') \in \mMC_0(\map(B,B))
\end{equation}
is connected to $\id'_B$. 

Proposition \ref{prop:cU-F-q-iso} implies that the $\sLie$-morphism
$$
\cU_{A B A} : \map(B,A) \to \map(A,A)
$$
is a quasi-isomorphism and, moreover, it satisfies 
the necessary conditions of Theorem 2.2 from \cite{GMtheorem}.
This theorem, in turn, implies that $\cU_{A B A}$  induces a bijection of sets 
\begin{equation}
\label{cU-ABA-pi0}
\pi_0\Big( \mMC_{\bul}(\map(B,A)) \Big)
 \xto{\cong}
\pi_0\Big( \mMC_{\bul}(\map(A,A)) \Big) \,.
\end{equation}

Therefore, there exists a MC element $G' \in \map(B,A)$ such 
that the $0$-cell 
\begin{equation}
\label{GF-pr}
(G \circ F)'= (\cU_{A B A})_* (G')
\end{equation}
is connected to $\id'_A$\,.

To prove that the $0$-cell $(F\circ G)' = (\cU_{B A B})_* (G')$ is connected to $\id'_B$, 
we consider the $\sLie$-morphisms 
\begin{equation}
\label{ABB}
\cU_{A  B  B} :  \map(B, B)  \to \map(A, B)\,,
\end{equation}
and
\begin{equation}
\label{AAB}
\cU_{AAB} :  \map(A, A)  \to \map(A, B)
\end{equation}
constructed, as above, using the $\infty$-morphism $F$. 

Proposition \ref{prop:cU-star-comp} implies that, if $K'$ is a MC element
of $\map(B, B)$ (resp.  $\map(A, A)$) corresponding to an $\infty$-morphism 
$K$ from $B$ to $B$ (resp. from $A$ to $A$) then the MC element 
$(\cU_{A B B})_*(K')$ (resp. $(\cU_{A A B})_*(K')$)
corresponds to the composition $K \circ F$ (resp. $F \circ K$).

Let us now consider the composition 
$F \circ  G$ and denote by $(F\circ G)'$ the $0$-cell of $\mMC_{\bul}(\map(B,B))$ corresponding 
to the $\infty$-morphism $F \circ G$. 

Due to the above observation, the MC element 
$(F\circ G \circ F)' \in \map(A, B)$ corresponding to the $\infty$-morphism 
$F\circ G \circ F$ satisfies the equations
\begin{equation}
\label{MC-FGF-FG}
(F\circ G \circ F)'  = (\cU_{A B B})_*\, (F\circ G)'
\end{equation}
and 
\begin{equation}
\label{MC-FGF-GF}
(F\circ G \circ F)'  = (\cU_{A A B})_* \, (G\circ F)'\,.
\end{equation}  

Therefore, since the $0$-cell $(G \circ F)'$ is connected to the 
$0$-cell $\id'_A$, the $0$-cell $(F\circ G \circ F)'$ is connected 
to $F'$ in $\mMC_{\bul} (\map(A,B))$\,.  
 
On other hand, $F' = (\cU_{A B B})_* (\id'_B)$, and hence 
the $0$-cells 
$$
(\cU_{A B B })_* (\id'_B) \qquad \textrm{and} \qquad  (\cU_{A B B})_*\, (F\circ G)'
$$
are connected in  $\mMC_{\bul} (\map(A,B))$\,.

Since $F$ is an $\infty$ quasi-isomorphism,  Proposition \ref{prop:cU-F-q-iso} and 
\cite[Theorem 2.2]{GMtheorem} imply that $(\cU_{A B B })_*$ induces 
a bijection 
$$
\pi_0 \Big( \mMC_{\bul} (\map(B,B)) \Big) \xto{\cong}  \pi_0 \Big( \mMC_{\bul} (\map(A,B)) \Big)\,.
$$

Thus the $0$-cells $(F\circ G)'$ and $\id'_B$ are also connected in $\mMC_{\bul} (\map(B,B))$
and the existence of a desired $\infty$-morphism $G$ is proved. 

Let $\wt{G}$ be another $\infty$-morphism from $B$ to $A$ such 
that $0$-cells \eqref{GF-pr} and $\id'_A$ are connected in 
$\mMC_{\bul}( \map(A,A) )$.  

Therefore, since $\cU_{ABA}$ induces bijection \eqref{cU-ABA-pi0} 
and the $0$-cells $(\cU_{ABA})_* (G')$ and $\id'_A$ are connected, 
we conclude that the $0$-cells $\wt{G}'$ and $G'$ are also connected.  
\end{proof}  

Thus we proved the first part of Theorem \ref{thm:pi-0-HoAlg}.  

We would like to conclude this subsection with the observation that 
mapping spaces of the simplicial category $\HoAlg^{\sD}_{\cC}$ enjoy 
the following remarkable property
\begin{cor}
\label{cor:q-iso-gives}
Let $A_1, A_2, A_3$ be $\Cobar(\cC)$-algebras and $F$ be 
an $\infty$ quasi-isomorphism from $A_1$ to $A_2$. Then the composition 
(resp. pre-composition) with $F$ induces the weak equivalences of simplicial sets: 
$$
\mMC_{\bul}(\map(A_3, A_1)) \to \mMC_{\bul}(\map(A_3, A_2))\,,
$$
$$
\mMC_{\bul}(\map(A_2, A_3)) \to \mMC_{\bul}(\map(A_1, A_3))\,.
$$
\end{cor}
\begin{proof}
The desired statement is a direct consequence of Proposition \ref{prop:cU-F-q-iso}
and \cite[Theorem 2.2]{GMtheorem}.
\end{proof}

\subsection{The functor $\mF$ from Theorem \ref{thm:pi-0-HoAlg} has the desired universal property}

The proof of the universal property of the functor $\mF$ is based on
the following proposition: 
\begin{prop}
\label{prop:homot-implies-OK}
Let $\mG : \Cat_{\cC} \to \mD$ be a functor which sends $\infty$ quasi-isomorphisms 
to isomorphisms in $\mD$. Let $A$, $B$ be $\Cobar(\cC)$-algebras and $F, G$ be 
$\infty$-morphisms from $A$ to $B$. If the corresponding 
$0$-cells $F'$ and $G'$ of $\mMC_{\bul}(\map(A,B))$ are connected 
then 
$$
\mG(F) = \mG(G)\,.
$$
\end{prop}
\begin{proof}
By the condition of the proposition, there exists a $1$-cell 
\begin{equation}
\label{K-pr}
K'  \in \Hom(\cC(A), B) \hotimes \bbk[t, \, dt]
\end{equation}
such that 
\begin{equation}
\label{proj-0}
K' \Big|_{t=d t = 0} = F'
\end{equation}
and 
\begin{equation}
\label{proj-1}
K' \Big|_{t=1, ~ d t = 0} = G'\,.
\end{equation}

Since 
$$
\cC(A) = \bigoplus_n \big( \cC(n) \otimes A^{\otimes \, n} \big)_{S_n}
$$
and $\Hom \big(\cC(A), B \big)$ is considered with the
topology coming from filtration \eqref{Conv-filter},
we have the natural strict $\sLie$-morphism
\begin{equation}
\label{isom-hotimes}
\Hom \big(\cC(A), B \big)  \hotimes \bbk[t, \, dt]  ~\to~
\Hom \big( \cC(A), B \otimes  \bbk[t, \, dt] \big) \,.
\end{equation}

Therefore the $1$-cell $K'$ gives us an $\infty$-morphism 
$K$ from $A$ to $B \otimes   \bbk[t, \, dt] $ which fits into 
the following commutative diagram 
\begin{equation}
\label{diag-K}
\begin{tikzpicture}
\matrix (m) [matrix of math nodes, row sep=2em, column sep=4em]
{   ~& ~  & B \\
A & B \otimes   \bbk[t, \, dt] & ~ \\
~ & ~ & B\,,\\ };
\path[->, font=\scriptsize]
(m-2-1) edge node[above] {$K$} (m-2-2) 
edge[bend left=15] node[above] {$F$} (m-1-3)  edge[bend right=15] node[below] {$G$} (m-3-3)
(m-2-2) edge node[above] {$p_0$} (m-1-3)  edge node[below] {$p_1$} (m-3-3);
\end{tikzpicture}
\end{equation}
where $ B \otimes   \bbk[t, \, dt]$ is considered with the differential 
$\pa_{B} + dt \,\pa_t $ and the natural $\Cobar(\cC)$-structure coming 
from the one on $B$. Moreover,  $p_0$ and $p_1$ are the obvious (strict)
morphisms of $\Cobar(\cC)$-algebras
\begin{equation}
\label{p-0}
p_0 (v) : = v \big|_{t=dt=0}  ~:~  B \otimes   \bbk[t, \, dt] \to B 
\end{equation}
and
\begin{equation}
\label{p-1}
p_1 (v) : = v \big|_{t=1,~ dt=0}  ~:~  B \otimes   \bbk[t, \, dt] \to B\,. 
\end{equation}

Let us observe that the maps $p_0$ and $p_1$ fit into the 
commutative diagram 
\begin{equation}
\label{diag-p0p1}
\begin{tikzpicture}
\matrix (m) [matrix of math nodes, row sep=2em, column sep=4em]
{   ~& ~  & B \\
B & B \otimes   \bbk[t, \, dt] & ~ \\
~ & ~ & B\,,\\ };
\path[->, font=\scriptsize]
(m-2-1) edge node[above] {$~~i$} (m-2-2) 
edge[bend left=15] node[above] {$\id_B$} (m-1-3)  edge[bend right=15] node[below] {$\id_B$} (m-3-3)
(m-2-2) edge node[above] {$p_0$} (m-1-3)  edge node[below] {$p_1$} (m-3-3);
\end{tikzpicture}
\end{equation}
where $i: B \to  B \otimes   \bbk[t, \, dt]$ is the natural embedding given by 
$$
i (v) : = v \otimes 1\,.
$$

Applying the functor $\mG$ to \eqref{diag-p0p1}, we get 
$$
\mG(p_0) \circ \mG(i) = \mG(p_1) \circ \mG(i) = \id_{\mG(B)}\,.
$$
Hence, since $i$ is obviously a quasi-isomorphism, we deduce that
$$
\mG(p_0) = \mG(p_1)\,.
$$

Finally, applying  the functor $\mG$ to \eqref{diag-K}, we get 
$$
\mG(p_0) \circ \mG(K) = \mG(F)\,, \qquad 
\mG(p_1) \circ \mG(K) = \mG(G)
$$  
which implies that $\mG(F) = \mG(G)$\,.
\end{proof}

Proposition \ref{prop:homot-implies-OK} motivates the following 
definition\footnote{For several other justifications of this definition, we refer the 
reader to paper \cite{DotsPoncin} by V. Dotsenko and N. Poncin.}
\begin{defi}
\label{dfn:homotopy}
Let $A,B$ be $\Cobar(\cC)$-algebras. We say that $\infty$-morphisms $F,G$ 
from $A$ to $B$ are \emph{homotopic} if the corresponding 
$0$-cells $F'$ and $G'$ are connected in $\mMC_{\bul}(\map(A,B))$.  
\end{defi}   
 
We can now prove the universal property of functor \eqref{mF}.
 
Indeed, let $\mG$ be a functor from $\Cat_{\cC}$ to some category $\mD$ 
which sends $\infty$ quasi-isomorphisms to isomorphisms.

For objects of $A,B, \dots$ of $\pi_0 \big( \HoAlg^{\sD}_{\cC} \big)$ we set 
$$
\mG'(A) : = \mG(A)\,.
$$

Next, given two $\Cobar(\cC)$-algebras $A,B$ and
an isomorphism class 
$$
[F'] \in \pi_0 \big( \mMC_{\bul}(\map(A,B)) \big) 
$$
of a MC element $F'$ in $\map(A,B)$ we set 
\begin{equation}
\label{mG-pr-morph}
\mG([F']) : = \mG(F)\,,
\end{equation}
where $F$ is the $\infty$-morphism from $A$ to $B$ corresponding 
to the MC element $F'$. 
 
Due to Proposition \ref{prop:homot-implies-OK}, the right hand side of 
\eqref{mG-pr-morph} does not depend on the choice of the MC element 
$F'$ in its isomorphism class $[F']$. 

It is clear that, this way, we get a functor 
$$
\mG' : \pi_0 \big( \HoAlg^{\sD}_{\cC} \big) \to \mD
$$
satisfying $\mG' \circ \mF = \mG$\,. 

It is also clear that such a functor $\mG'$ is unique and the 
proof of Theorem \ref{thm:pi-0-HoAlg} is complete. 
We can use the above results to obtain the following converse to
Cor.\ \ref{cor:q-iso-gives}. Together they give a
recognition principal for $\infty$ quasi-isomorphisms, in analogy
with the characterization of weak equivalences via function complexes in a simplicial model category.

\begin{cor}
\label{cor:q-iso-gives-converse}
For every $\infty$-morphism $F \maps A \to B$ of $\Cobar(\cC)$
algebras, the following statements are equivalent:
\begin{enumerate}

\item  $F$ is an $\infty$-quasi-isomorphism. 

\item The $\sLie$-morphisms
\[
\cU_{AAB} \maps \map(A,A) \to \map(A,B)  \quad \textrm{ and } \quad
\cU_{BAB} \maps \map(B,A) \to \map(B,B)
\]
induce homotopy equivalences of simplicial sets
\[
\mMC_{\bul}(\map(A, A)) \xto{\sim} \mMC_{\bul}(\map(A, B)) 
~ \textrm{ and } ~
\mMC_{\bul}(\map(B, A)) \xto{\sim} \mMC_{\bul}(\map(B, B))
\]
respectively. 
 
\item  The $\sLie$-morphisms
\[
\cU_{ABA} \maps \map(B,A) \to \map(A,A) 
\quad \textrm{ and } \quad
\cU_{ABB} \maps \map(B, B) \to \map(A, B) 
\]
induce homotopy equivalences of simplicial sets
\[
\mMC_{\bul}(\map(B, A)) \xto{\sim} \mMC_{\bul}(\map(A, A)) 
~ \textrm{ and } ~
\mMC_{\bul}(\map(B, B)) \xto{\sim} \mMC_{\bul}(\map(A, B))
\]
respectively. 
\end{enumerate}
\end{cor}

\begin{proof}
The implications $1 \Rightarrow 2$ and $1 \Rightarrow 3$ are particular 
cases of  Corollary \ref{cor:q-iso-gives}. Let us prove the implication $2 \Rightarrow 1$. 

Since $F$ induces a homotopy equivalence of simplicial sets
\[
\mMC_{\bul}(\map(B, A)) \xto{\sim} \mMC_{\bul}(\map(B, B)),
\]
we have an isomorphism
\[
\pi_0 \Bigl(\mMC_{\bul}(\map(B, A))\Bigr ) \cong \pi_0
\Bigl(\mMC_{\bul}(\map(B, B)) \Bigr),
\]
which implies that there exists an $\infty$-morphism $G \maps B \to A$
such that $F \circ G$ and $\id_B$ are homotopic, in the sense of Def.\ \ref{dfn:homotopy}.
As in the proof of Prop.\ \ref{prop:homot-implies-OK}, this gives a
commutative diagram of $\Cobar(\cC)$ algebras and $\infty$-morphisms:
\begin{equation*}
\begin{tikzpicture}
\matrix (m) [matrix of math nodes, row sep=2em, column sep=4em]
{   ~& ~  & B \\
B & B \otimes   \bbk[t, \, dt] & ~ \\
~ & ~ & B\,\\ };
\path[->, font=\scriptsize]
(m-2-1) edge node[above] {$H$} (m-2-2) 
edge[bend left=15] node[above] {$F\circ G$} (m-1-3)  edge[bend right=15] node[below] {$\id_B$} (m-3-3)
(m-2-2) edge node[above] {$p_0$} (m-1-3)  edge node[below] {$p_1$} (m-3-3);
\end{tikzpicture}
\end{equation*}
By taking the linear terms of the $\infty$-morphisms, this diagram, in turn, gives a diagram of cochain complexes
\begin{equation*}
\begin{tikzpicture}
\matrix (m) [matrix of math nodes, row sep=2em, column sep=4em]
{   ~& ~  & B \\
B & B \otimes   \bbk[t, \, dt] & ~ \\
~ & ~ & B\,,\\ };
\path[->, font=\scriptsize]
(m-2-1) edge node[above] {$h$} (m-2-2) 
edge[bend left=15] node[above] {$f\circ g$} (m-1-3)  edge[bend right=15] node[below] {$\id_B$} (m-3-3)
(m-2-2) edge node[above] {$p_0$} (m-1-3)  edge node[below] {$p_1$} (m-3-3);
\end{tikzpicture}
\end{equation*}
where $h$ and $f\circ g$ are the linear terms $p_{B\tensor \bbk[t,dt]}
\circ H \vert_{B}$, and $p_{B} \circ F\circ G \vert_{B}$, respectively.
Let $I \maps B\tensor \bbk[t,dt] \to B$ denote  ``integration over the fiber'', i.e.
the degree $-1$ linear map
\[
I\Bigl( b \tensor q(t) + \tilde{b} \tensor \tilde{q}(t)dt \Bigr) = (-1)^{\vert
  \tilde{b} \vert} ~ \tilde{b}\int^{1}_{0} \tilde{q}(t)dt.
\]
Then one can show that the map $s_B \maps B \to B$ defined as
\[
s_B : = I \circ h
\]
is a chain homotopy
\[
\id_B - f \circ g  = \partial_B s_B + s_B \partial_B.
\]
Hence, the map $f$ is surjective on cohomology. 

Now consider the $\sLie$-morphism
\[
\cU_{ABB} \maps \map(B,B) \to \map(A,B) 
\]
which sends a MC element $K'$ to $(K \circ F)'$. This gives a map
between sets
\[
\pi_0 \Bigl(\mMC_{\bul}(\map(B, B))\Bigr ) \to \pi_0\Bigl(\mMC_{\bul}(\map(A, B)) \Bigr).
\]
Since the MC elements $(F \circ G)'$ and $(\id_B)'$ are connected by a $1$-cell, 
we conclude that $F'$ is connected to $(F \circ G \circ F)'$. 
By the first part of statement 2, the $\sLie$-morphism
\[
\cU_{AAB} \maps \map(A,A) \to \map(A,B) 
\]
induces a bijection
\[
\pi_0 \Bigl(\mMC_{\bul}(\map(A, A))\Bigr ) \cong \pi_0
\Bigl(\mMC_{\bul}(\map(A, B)) \Bigr).
\]
sending $[(\id_A)']$ to $[F']$,
and $[(G \circ F)']$ to $[(F \circ G \circ F)']=[F']$. Hence, we
see that $\id_A$ is homotopic to $G \circ F$. We produce a chain
homotopy $s_A \maps A \to A$  using the same construction as before,
which shows that $f$ is also injective on cohomology. Hence, $F \maps A \to
B$ is an $\infty$-quasi-isomorphism.

The proof of the implication $3 \Rightarrow 1$ is very similar to that of $2 \Rightarrow 1$. 
So we leave it to the reader.
\end{proof}

\section{The Homotopy Transfer Theorem is a simple consequence of the Goldman-Millson theorem}
\label{sec:HTT}

In this section we give an elegant proof of the Homotopy Transfer Theorem
for $\Cobar(\cC)$-algebras. This proof is based on a construction from \cite{DefHomotInvar}
and a version of the Goldman-Millson theorem from  \cite{GMtheorem}.

Let us consider a dg cooperad $\cC$ for which $\cC_{\c}$ carries 
ascending filtration \eqref{cC-circ-filtr} satisfying condition \eqref{cocomplete}
and let $A$, $B$ be cochain complexes. 

In  \cite[Section 3.1]{DefHomotInvar}, we equipped the cochain complex 
\begin{equation}
\label{Cyl-cC-A-B}
\Cyl(\cC, A, B) : = \bsi \Hom(\cC_{\c}(A), A) ~\oplus~ \Hom(\cC(A), B)~ \oplus~ 
 \bsi \Hom(\cC_{\c}(B), B)
\end{equation}
with a natural $\sLie$-algebra structure\footnote{In \cite{DefHomotInvar}, we actually 
introduce an $L_{\infty}$-structure on the suspension of \eqref{Cyl-cC-A-B}. But the latter is,
of course, equivalent to introducing a $\sLie$-algebra structure on  \eqref{Cyl-cC-A-B}.} 

According to \cite[Section 3.1]{DefHomotInvar}, MC elements of $\Cyl(\cC, A, B)$ are 
triples: 
\begin{itemize}

\item a $\Cobar(\cC)$-algebra structure on $A$, 

\item a $\Cobar(\cC)$-algebra structure on $B$, and 

\item an $\infty$-morphism from $A$ to $B$\,.

\end{itemize}

In particular, any chain map $\vf : A \to B$ gives a MC element $Q_{\vf}$
corresponding to the zero  $\Cobar(\cC)$-algebra structures on $A$, $B$, 
and the strict $\infty$-morphism from $A$ to $B$.

Let us twist the $\sLie$-algebra structure on  \eqref{Cyl-cC-A-B}, 
and denote the new $\sLie$-algebra by 
\begin{equation}
\label{Cyl-cC-AB-twisted}
\Cyl(\cC, A, B)^{Q_{\vf}}\,.
\end{equation}

It is not hard to see that the graded subspace
\begin{equation}
\label{Cyl-cC-0-AB}
\Cyl_{\c}(\cC, A, B)^{Q_{\vf}} : = \bsi \Hom(\cC_{\c}(A), A) ~\oplus~ \Hom(\cC_{\c}(A), B)~ \oplus~ 
 \bsi \Hom(\cC_{\c}(B), B)
\end{equation}
is a $\sLie$-subalgebra of \eqref{Cyl-cC-AB-twisted} and 
filtration  \eqref{cC-circ-filtr} on $\cC_{\c}$ allows us to equip 
\eqref{Cyl-cC-0-AB} with a natural complete descending filtration 
$\cF_{\bul} \Cyl_{\c}(\cC, A, B)^{Q_{\vf}}$ (see Remark 2 in \cite[Section 3.2]{DefHomotInvar})
such that 
\begin{equation}
\label{Cyl-cC-0-AB-OK} 
\Cyl_{\c}(\cC, A, B)^{Q_{\vf}} = \cF_{1} \Cyl_{\c}(\cC, A, B)^{Q_{\vf}}\,.
\end{equation}
In other words, \eqref{Cyl-cC-0-AB} is a filtered $\sLie$-algebra.

Furthermore, according to \cite[Section 3.2]{DefHomotInvar},  
MC elements of \eqref{Cyl-cC-0-AB} are in bijection with triples: 
\begin{itemize}

\item a $\Cobar(\cC)$-algebra structure on $A$, 

\item a $\Cobar(\cC)$-algebra structure on $B$, and 

\item an $\infty$-morphism $F$ from $A$ to $B$ whose linear term is $\vf$.

\end{itemize}

The Homotopy Transfer Theorem can be now formulated as follows: 
\begin{thm}
\label{thm:HTT}
Let $B$ be a $\Cobar(\cC)$-algebra, $A$ be a cochain complex 
and $\vf : A \to B$ be a quasi-isomorphism of cochain complexes. 
Then there exists a  $\Cobar(\cC)$-algebra structure $Q_A$ on $A$, 
a $\Cobar(\cC)$-algebra structure $Q_B$ on $B$ (which is homotopy 
equivalent to the original one) and an $\infty$-morphism $F$ from 
$(A, Q_A)$ to $(B,Q_B)$ whose linear term is $\vf$. If 
$(\ti{Q}_A, \ti{Q}_B, \ti{F})$ is another triple
satisfying the above properties then the MC elements 
corresponding to $(Q_A, Q_B, F)$ and  $(\ti{Q}_A, \ti{Q}_B, \ti{F})$
are isomorphic in 
$$
\mMC_{\bul} \big(\, \Cyl_{\c}(\cC, A, B)^{Q_{\vf}}\, \big)\,.
$$
\end{thm}
\begin{proof}
Let us identify the graded vector space of 
the convolution Lie algebra 
$$
\Conv(\cC_{\c}, \End_B)
$$
with $\Hom(\cC_{\c}(B), B)$ and consider $\bsi \Hom(\cC_{\c}(B), B)$ with 
the corresponding $\sLie$-algebra structure.  

Due to \cite[Proposition 3.2]{DefHomotInvar}, the canonical projection 
\begin{equation}
\label{pi-B}
\pi_B ~:~   \Cyl_{\c}(\cC, A, B)^{Q_{\vf}} \to \bsi \Hom(\cC_{\c}(B), B)
\end{equation}
is a strict quasi-isomorphism of $\sLie$-algebras which is 
obviously compatible with the descending filtrations coming 
from  \eqref{cC-circ-filtr}.

Using the same arguments, as in the proof of  \cite[Proposition 3.2]{DefHomotInvar},
it is easy to see that 
$$
\pi_B ~:~   \cF_m \Cyl_{\c}(\cC, A, B)^{Q_{\vf}} \to \bsi \cF_m \Hom(\cC_{\c}(B), B) 
$$
is a quasi-isomorphism of cochain complexes for every $m \ge 1$. 

Therefore, applying Theorem 1.1 from \cite{GMtheorem} to \eqref{pi-B}, 
we conclude that $\pi_B$ induces a weak equivalence of simplicial sets 
$$
\mMC_{\bul} \big(\, \Cyl_{\c}(\cC, A, B)^{Q_{\vf}}\, \big) \to 
\mMC_{\bul} \big( \bsi \Hom(\cC_{\c}(B), B) \big)
$$
and hence a bijection
$$
\pi_0\Big( \mMC_{\bul} \big(\, \Cyl_{\c}(\cC, A, B)^{Q_{\vf}}\, \big) \Big) \to 
\pi_0\Big( \mMC_{\bul} \big( \bsi \Hom(\cC_{\c}(B), B) \big) \Big)\,.
$$

Thus Theorem \ref{thm:HTT} is a simple consequence of the fact 
that MC elements of the $\sLie$-algebra 
\begin{equation}
\label{bsi-Hom-cCB-B}
\bsi \Hom(\cC_{\c}(B), B)
\end{equation}
are in bijection with 
$\Cobar(\cC)$-algebra structures on $B$. Moreover, homotopy equivalent 
$\Cobar(\cC)$-algebra structures on $B$ correspond precisely to isomorphic 
MC elements of \eqref{bsi-Hom-cCB-B}.
\end{proof}

\begin{remark}
\label{rem:HTT}
In the usual version of the Homotopy Transfer Theorem (see  \cite[Theorem 10.3.2]{LV}) 
one constructs a $\Cobar(\cC)$-algebra structure on $A$ and an $\infty$ quasi-isomorphism 
$F$ from $A$ to $B$ with the original $\Cobar(\cC)$-algebra structure, while in the 
above theorem, $F$ lands in $(B, Q_B)$ where $Q_B$ is only homotopy equivalent to 
the original one. On the other hand, given an isomorphism connecting two MC elements
$\ti{Q}_B$ and $Q_B$ in 
$$
\mMC_{\bul} \big( \bsi \Hom(\cC_{\c}(B), B) \big)\,,
$$
it is easy to construct an $\infty$-morphism $G$ from $(B, Q_B)$ to 
$(B, \ti{Q}_B)$ whose linear term is $\id_B$\,. So the ``usual'' version of 
the Homotopy Transfer Theorem follows from Theorem \ref{thm:HTT}. 
\end{remark}

\appendix

\section{Proof of Proposition \ref{prop:Conv-V-A}}
\label{app:proof-Conv}

Although very similar claims to Proposition \ref{prop:Conv-V-A}
appeared in the literature (see, for example, \cite{Berglund}, 
\cite{FMYau}, \cite{KhPQ}, \cite{MVnado1}) we still decided to give its proof
for convenience of the reader.  

Let us denote by $Q'_A$ the composition 
$$
p_A \circ Q_A : \cC(A) \to A\,.
$$

To prove that multi-bracket \eqref{m-bracket} is symmetric 
in its arguments, we let
\begin{equation}
\label{Delta-m}
\D_m(v) = \sum_{\al} (\ga_{\al}; v^{\al}_1, v^{\al}_{2}, \dots, v^{\al}_m)
\end{equation}
and recall that $\D_m(v)$ lands in $S_m$-invariants of $ \cC(m) \otimes V^{\otimes\, m} $.
In other words, for every $\si \in S_m$ we have 
\begin{equation}
\label{S-m-invar}
\sum_{\al} \big( \si^{-1}(\ga_{\al}); v^{\al}_{\si(1)}, v^{\al}_{\si(2)}, \dots, v^{\al}_{\si(m)} \big) =
 \sum_{\al} (\ga_{\al}; v^{\al}_1, v^{\al}_{2}, \dots, v^{\al}_m)\,.
\end{equation}

Therefore, for every $\si \in S_m$, we have 
\begin{multline*}
\{f_{\si(1)}, \dots, f_{\si(m)}\} (v) = \\
\sum_{\al} Q'_A \circ (1 \otimes f_{\si(1)}\otimes \dots \otimes f_{\si(m)} ) 
 \big( \si^{-1}(\ga_{\al}); v^{\al}_{\si(1)}, v^{\al}_{\si(2)}, \dots, v^{\al}_{\si(m)} \big)=\\
\sum_{\al} \pm\, Q'_A \big( \si^{-1}(\ga_{\al});   f_{\si(1)}(v^{\al}_{\si(1)}), 
f_{\si(2)}(v^{\al}_{\si(2)}), \dots, f_{\si(m)}(v^{\al}_{\si(m)}) \big)\,.
\end{multline*}

Thus, since $Q_A$ is compatible with the action of the symmetric group, 
\begin{multline*}
\{f_{\si(1)}, \dots, f_{\si(m)}\} (v) =\\
\sum_{\al} \pm\, Q'_A \big( \si^{-1}(\ga_{\al});   f_{\si(1)}(v^{\al}_{\si(1)}), 
f_{\si(2)}(v^{\al}_{\si(2)}), \dots, f_{\si(m)}(v^{\al}_{\si(m)}) \big) =\\
\sum_{\al} \pm\, Q'_A \big( \ga_{\al};   f_{1}(v^{\al}_{1}), 
f_{2}(v^{\al}_{2}), \dots, f_{m}(v^{\al}_{m}) \big) =\\
 \sum_{\al} \pm Q'_A \circ (1 \otimes f_{1}\otimes \dots \otimes f_{m} )
 (\ga_{\al}; v^{\al}_1, v^{\al}_{2}, \dots, v^{\al}_m) =
\pm \{f_{1}, \dots, f_{m}\} (v)\,. 
\end{multline*}
So multi-bracket \eqref{m-bracket} is indeed symmetric\footnote{In the above calculations, 
the sign factors can be easily deduced from the Koszul rule of signs.} 
in its arguments.

Let us now prove that the operation
\begin{equation}
\label{diff-Conv}
f \mapsto \{f\}
\end{equation}
is a differential on  $\Hom(V, A)$, i.e. $\{\{f\}\} =0$ for every 
$f \in  \Hom(V, A)$. 

Indeed, using the identities $\pa^2_A =0$, $\pa^2_{V} = 0$
and the compatibility of $\D_1$ with the differentials 
on $V$ and $\cC(1) \otimes V$, we get 
\begin{multline}
\label{1-brack-square}
\{\{f\}\}(v) = \pa_{A}\, \{f\}(v) -(-1)^{|f|+1} \{f\}(\pa_{V}\, v)  +  Q'_A\big( 1 \otimes \{f\} (\D_1(v)) \big) =\\
\pa_A \circ Q'_A \circ (1 \otimes f) \circ \D_1(v) + 
(-1)^{|f|} Q'_A \circ (1 \otimes f) \big( (\pa_{\cC} + \pa_V) \D_1(v) \big) \\
+ Q'_A \circ (1\otimes (\pa_A \circ f)) \big(\D_1 (v) \big)
-(-1)^{|f|} Q'_A \circ (1\otimes (f\circ \pa_V)) \big(\D_1 (v) \big) \\ 
+ Q'_A \circ (1 \otimes Q'_A) \circ (1\otimes 1 \otimes f) \big( (1\otimes \D_1)\circ \D_1 (v) \big)
\end{multline}
for every $v \in V$\,. 
 
Using the axioms of a coalgebra over a cooperad, we rewrite the 
expression $(1\otimes \D_1)\circ \D_1 (v) $ in \eqref{1-brack-square} as 
follows
\begin{equation}
\label{D1D1}
(1\otimes \D_1)\circ \D_1 (v) = (\D^{\cC}  \otimes 1) \circ \D_1 (v)\,,
\end{equation}
where $\D^{\cC}$ is the cooperadic comultiplication
$$
\D^{\cC} : \cC(1) \to \cC(1) \otimes \cC(1)\,.
$$

Therefore the expression $\{\{f\}\}(v)$ can be rewritten as 
\begin{equation}
\label{1-brack-sq-easy}
\{\{f\}\}(v) =  \big(\pa_{A} \circ Q'_A + Q'_A \circ (\pa_{\cC} + \pa_A) + Q'_A \circ Q_A 
\big) \circ (1 \otimes f) \circ \D_1 (v) \,.
\end{equation}

Thus the identity $ \{\{f\}\} =0 $ is a consequence of the 
MC equation for $Q'_A$. 

Our goal now is to prove the relation 
\begin{equation}
\label{goal-brack}
\sum_{p=1}^{m} 
\sum_{\si \in \Sh_{p, m-p} } 
(-1)^{\ve(\si; f_1, \dots, f_m)}
\{ \{f_{\si(1)}, \dots, f_{\si(p)} \}, f_{\si(p+1)}, \dots, f_{\si(m)} \} (v) =0
\end{equation}
for every $m \ge 2$ and $v \in V$\,, where the sign factor 
$(-1)^{\ve(\si; f_1, \dots, f_m)}$ is defined in \eqref{ve-si-vvv}. 

In our calculations below, we often put $\pm$ instead of 
the precise sign factor. These sign factors can be easily deduced 
from the Koszul rule of signs.

Unfolding \eqref{goal-brack} we get
\begin{multline}
\label{goal-unfolded}
\sum_{p=1}^{m} 
\sum_{\si \in \Sh_{p, m-p} } 
(-1)^{\ve(\si; f_1, \dots, f_m)}
\{ \{f_{\si(1)}, \dots, f_{\si(p)} \}, f_{\si(p+1)}, \dots, f_{\si(m)} \} (v) = \\
\pa_{A} \big( \{f_1, \dots, f_m\} (v)\big) + (-1)^{|f_1| + \dots + |f_m|} 
 \{f_1, \dots, f_m\} \big(\pa_{V}(v) \big) \\
+ \sum_{i=1}^{m} (-1)^{|f_1| + \dots + |f_{i-1}|} \{f_1, \dots, (\pa_{A} \circ f_i - (-1)^{|f_i|} f_i \circ \pa_{V} ), \dots,  f_m\} (v)\\
+ \sum_{ \substack{1 \le  p \le  m\\[0.1cm] 
\si \in \Sh_{p, m-p} } } (-1)^{\ve(\si; f_1, \dots, f_m)}
 Q'_{A}\big(1; Q'_{A}\big( (1; f_{\si(1)}, \dots, f_{\si(p)}) \D_{p}(-)\big) , f_{\si(p+1)}, \dots, f_{\si(m)} (\D_{m-p+1}(v))\big)\,.
\end{multline}

Expanding the expression $\pa_{A} \big( \{f_1, \dots, f_m\} (v)\big)$, we get
\begin{eqnarray*}
\pa_{A} \big( \{f_1, \dots, f_m\} (v)\big) &=& \pa_{A}Q'_A \big( 1; f_{1},\dots, f_{m}  
(\D_m (v))  \big)\\
&=& \sum_{\al}\pa_{A}Q'_A \big( (1; f_{1}, \dots, f_{m}  )
( \ga_{\al}; v^{\al}_1, v^{\al}_{2}, \dots, v^{\al}_m))\\
&=& \sum_{\al}\pm\pa_{A}Q'_A \big( \ga_{\al}; f_{1}(v^{\al}_1), \dots, f_{m}(v^{\al}_m)).
\end{eqnarray*}

Using
$$
\D_m(\pa_V(v)) = 
\sum_{\al} (\pa_{\cC}(\ga_{\al}); v^{\al}_1, v^{\al}_{2}, \dots, v^{\al}_m) + 
\sum_{ \substack{\al \\[0.1cm] 1\leq i\leq m}}\pm(\ga_{\al}; v^{\al}_1, \dots, \pa_V(v^{\al}_{i}),\dots, v^{\al}_m)\,,
$$
we expand $\{f_1, \dots, f_m\} \big(\pa_{V}(v) \big)$ obtaining
\begin{eqnarray*}
\{f_1, \dots, f_m\} \big(\pa_{V}(v) \big) &=& Q'_A \big( 1; f_{1}, \dots, f_{m}  
(\D_m (\pa_{V}(v)))  \big)\\
&=&\sum_{\al} Q'_A \big( (1; f_{1}, \dots, f_{m} ) 
(\pa_{\cC}(\ga_{\al}); v^{\al}_1, v^{\al}_{2}, \dots, v^{\al}_m)  \big)\\
&+&  \sum_{ \substack{\al \\[0.1cm] 1\leq i\leq m}}\pm Q'_A \big( (1; f_{1}, \dots, f_{m} )(\ga_{\al}; v^{\al}_1, \dots, \pa_V(v^{\al}_{i}),\dots, v^{\al}_m)\big)\\
&=&\sum_{\al} \pm Q'_A \big( \pa_{\cC}(\ga_{\al}); f_{1}(v^{\al}_1), \dots, f_{m}(v^{\al}_m) \big) \\
&+& \sum_{ \substack{\al \\[0.1cm] 1\leq i\leq m}} \pm Q'_A \big(\ga_{\al}; f_{1}(v^{\al}_1), \dots, f_i(\pa_V(v^{\al}_{i})),\dots, f_{m}(v^{\al}_m)\big).
\end{eqnarray*}

Expanding the sum 
$$
\sum_{i=1}^{m} (-1)^{|f_1| + \dots + |f_{i-1}|} \{f_1, \dots, (\pa_{A} \circ f_i - (-1)^{|f_i|} f_i \circ \pa_{V} ), \dots,  f_m\} (v)
$$
we obtain
\begin{multline}
\label{brack-f-df-fd-f}
\sum_{i=1}^{m} (-1)^{|f_1| + \dots + |f_{i-1}|} \{f_1, \dots, (\pa_{A} \circ f_i - (-1)^{|f_i|} f_i \circ \pa_{V} ), \dots,  f_m\} (v) = \\
 \sum_{ \substack{\al \\[0.1cm] 1\leq i\leq m}} \pm Q'_A \big( 1; f_{1},\dots,
  (\pa_{A} \circ f_i - (-1)^{|f_i|} f_i \circ \pa_{V} ), \dots, f_{m} \big)(\ga_{\al}; v^{\al}_1, v^{\al}_{2}, \dots, v^{\al}_m) \\
 = \sum_{ \substack{\al \\[0.1cm] 1\leq i\leq m}}\pm Q'_A \big( \ga_{\al}; f_{1}(v^{\al}_1),\dots,
   (\pa_{A} \circ f_i - (-1)^{|f_i|} f_i \circ \pa_{V} ) (v^{\al}_i), \dots, f_{m}(v^{\al}_m) \big)\\
 = \sum_{ \substack{\al \\[0.1cm] 1\leq i\leq m}}\pm Q'_A \big( \ga_{\al}; f_{1}(v^{\al}_1),\dots, \pa_{A} (f_i(v^{\al}_i)),\dots,f_{m}(v^{\al}_m) \big)\\
 - \sum_{ \substack{\al \\[0.1cm] 1\leq i\leq m}}\pm Q'_A \big( \ga_{\al};f_{1}(v^{\al}_1),\dots, f_i(\pa_{V}v^{\al}_i),\dots, f_{m}(v^{\al}_m) \big).
\end{multline}

The last sum in the R.H.S. of \eqref{goal-unfolded} is expanded as follows:
\begin{multline}
\label{QQ-fff}
\sum_{ \substack{1 \le  p \le  m\\[0.1cm] 
\si \in \Sh_{p, m-p} } } (-1)^{\ve(\si; f_1, \dots, f_m)}
 Q'_{A}\big(1; Q'_{A}\big( (1; f_{\si(1)}, \dots, f_{\si(p)}) \D_{p}(-)\big) , f_{\si(p+1)}, \dots, f_{\si(m)} (\D_{m-p+1}(v))\big) \\
= \sum_{ \substack{1 \le  p \le  m\\[0.1cm] 
\si \in \Sh_{p, m-p} } }\pm
Q'_A ( 1 \otimes Q'_A  \otimes 1^{\otimes (m-p)}) \big( 1; (1; f_{\si(1)}, \dots , f_{\si(p)}) 
, f_{\si(p+1)},\dots , f_{\si(m)} \big) \\
\big( (1 \otimes \D_p \otimes 1^{\otimes (m-p)} ) \circ  \D_{m-p+1} (v)  \big)\,.
\end{multline}

Using the axioms of $\cC$-coalgebra structure on $V$, we can rewrite the term 
$(1 \otimes \D_p \otimes 1^{\otimes (m-p)} )\circ \D_{m-p+1} (v) $
as follows:
\begin{equation}
\label{Del-Del-v}
(1 \otimes \D_p \otimes 1^{\otimes (m-p)} ) \circ \D_{m-p+1} (v) = 
(\D^{\cC}_{\bt_{m, p}} \otimes 1^{\otimes \, m})\circ \D_m (v)\,,
\end{equation}
where $\bt_{m,p}$ is the labeled planar tree depicted in Figure 
\ref{fig:bt-m-p-i}.
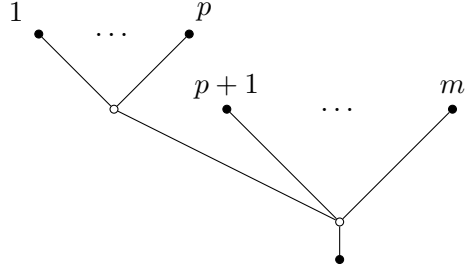
\begin{figure}[htp]
\centering 
\begin{tikzpicture}[scale=0.5]
\tikzstyle{w}=[circle, draw, minimum size=3, inner sep=1]
\tikzstyle{b}=[circle, draw, fill, minimum size=3, inner sep=1]
\node[w] (l1) at (0, 3) {};
\draw (6,3) node[anchor=center] {{\small $\dots$}};
\node[b] (limp) at (3, 3) {};
\draw (3,3.6) node[anchor=center] {{\small $p+1$}};
\node[b] (limp1) at (-2, 5) {};
\draw (-2.6,5.6) node[anchor=center] {{\small $1$}};
\draw (0,5) node[anchor=center] {{\small $\dots$}};
\node[b] (li) at (2, 5) {};
\draw (2.4,5.6) node[anchor=center] {{\small $p$}};
\node[b] (lm) at (9, 3) {};
\draw (9,3.6) node[anchor=center] {{\small $m$}};

\node[w] (v1) at (6, 0) {};
\node[b] (r) at (6, -1) {};
\draw (r) edge (v1);
\draw (v1) edge (l1);
\draw (v1) edge (limp);
\draw (v1) edge (lm);
\draw (l1) edge (limp1);
\draw (l1) edge (li);
\end{tikzpicture}
\caption{\label{fig:bt-m-p-i} The labeled planar tree $\bt_{m,p}$}
\end{figure} 
and let $\ga^{\si}_{\al, \beta, 1}$ and $\ga^{\si}_{\al, \beta, 2}$ be the tensor factors in 
\begin{equation}
\label{Del-tau-bt-mpi}
\D^{\cC}_{\si(\bt_{m,p})}(\ga_{\al}) = \sum_{\beta} \ga^{\si}_{\al, \beta, 1} \otimes \ga^{\si}_{\al, \beta, 2}\,,
\end{equation}
where $\si(\bt_{m,p})$ is the tree corresponding to the $(p,m-p)$-shuffle $\si$.

Using the axioms of the cooperadic comultiplication $\D^{\cC}$, we get 
$$
\textrm{Sum \eqref{QQ-fff}} =
$$
\begin{equation}
\label{the-sum111}
\sum^{\al, \beta}_{ \substack{1 \le  p \le  m \\[0.1cm] 
\si \in \Sh_{p, m-p} } } 
\pm Q'_A( 1 \otimes Q'_A \otimes 1^{\otimes (m-p)})
\end{equation}
$$
\Big(\, \ga^{\si}_{\al, \beta, 1}\,; ( \,\ga^{\si}_{\al, \beta, 2}; f_{\si(1)}(v^{\al}_{\si(1)}),\dots, f_{\si(p)}(v^{\al}_{\si(p)})), 
f_{\si(p+1)} (v^{\al}_{\si(p+1)}), \dots, 
f_{\si(m)} (v^{\al}_{\si(m)})  \Big)\,,
$$
which can be rewritten in terms of the bracket on 
the convolution Lie algebra $\Conv(\cC_{\c}, \End_{A})$ (Prop. 4.1, \cite{notes}).
Namely, 
$$
\textrm{Sum \eqref{QQ-fff} } =
$$
\begin{equation}
\label{the-sum-OK}
\frac{1}{2} \sum_{\al} [Q'_A , Q'_A] (\ga_{\al}; f_1(v^{\al}_1), \dots, f_m(v^{\al}_m))\,,
\end{equation}
where we tacitly identify the composition $Q'_A = p_A \circ Q_A$ with
the corresponding element in  
\begin{equation}
\label{Conv-cC-End-A}
\Conv(\cC_{\c}, \End_{A}) = \prod_{n \ge 1} 
\Hom_{S_n} \big(\cC_{\c}(n),   \End_{A}(n) \big)\,.
\end{equation}

Collecting the expanded terms of the R.H.S. of \eqref{goal-unfolded} we obtain 
\begin{equation*}
\sum_{\al}\pm\pa_{A}Q'_A \big( \ga_{\al}; f_{1}(v^{\al}_1), \dots, f_{m}(v^{\al}_m)) 
+ \sum_{\al} \pm Q'_A \big( \pa_{\cC}(\ga_{\al}); f_{1}(v^{\al}_1), \dots, f_{m}(v^{\al}_m) ) \\
\end{equation*}
\begin{equation}
\label{pa_V-1}
+ \sum_{ \substack{\al \\[0.1cm] 1\leq i\leq m}} \pm Q'_A \big(\ga_{\al}; f_{1}(v^{\al}_1), \dots, f_i(\pa_V(v^{\al}_{i})),  
\dots,  f_{m}(v^{\al}_m)\big)\\
\end{equation}
\begin{equation}
\label{pa_V-2}
 - \sum_{ \substack{\al \\[0.1cm] 1\leq i\leq m}}\pm Q'_A \big( \ga_{\al}; f_{1}(v^{\al}_1), \dots, f_i(\pa_{V}v^{\al}_i), \dots, f_{m}(v^{\al}_m) \big)\\
 \end{equation}
\begin{equation*}
 \sum_{ \substack{\al \\[0.1cm] 1\leq i\leq m}} \pm Q'_A \big( \ga_{\al}; f_{1}(v^{\al}_1), \dots, \pa_{A} (f_i(v^{\al}_i)), \dots, f_{m}(v^{\al}_m) \big) 
 + \frac{1}{2} \sum_{\al} [Q'_A , Q'_A] (\ga_{\al}; f_1(v^{\al}_1), \dots, f_m(v^{\al}_m))\,
\end{equation*}

Canceling terms \eqref{pa_V-1} and \eqref{pa_V-2}, we see that the right hand side 
of \eqref{goal-unfolded} can be rewritten as 
$$
\textrm{The R.H.S. of  \eqref{goal-unfolded} }  = 
$$
\begin{equation}
\Big( \pa_A  \circ Q'_A  + Q'_A \circ (\pa_{\cC} + \pa_A)   + \frac{1}{2} [Q'_A, Q'_A] \Big)  (1;f_1,\ldots,f_m)(\D_m(v))\,,
\label{RHS-goal-unfolded}
\end{equation}
where, as above, we identify $Q'_A$ with the corresponding element in 
\eqref{Conv-cC-End-A}.  

Thus desired equations \eqref{goal-brack} are satisfied due to the fact that 
$Q'_A$ is a MC element of dg Lie algebra \eqref{Conv-cC-End-A}. 

We conclude this proof by a comment about the sign factors. 

All sign factors in the above computations are subject to the usual Koszul rule.  
For example, to show explicitly that terms \eqref{pa_V-1} and \eqref{pa_V-2} cancel 
each other, we need to verify that the corresponding contributions from the term
\begin{equation}
\label{term1}
(-1)^{|f_1| + \dots + |f_m|}  \{f_1,\dots,f_m\}(\pa_{V}(v))
\end{equation}
matches with the contributions from  
\begin{equation}
\label{term2}
(-1)^{|f_1| + \dots + |f_i|} \{f_1,\dots,f_i \circ  \pa_{V},\dots,f_m\}(v)\,.
\end{equation}
in equation \eqref{goal-unfolded}. 

It is easy to see that the contribution 
\begin{equation}
\label{contribution}
Q'_{A}(\ga_{\al};f_1(v_{1}^{\al}),\dots,f_i \circ \pa_V(v_{i}^{\al}),\dots f_m(v_{m}^{\al}))
\end{equation}
from \eqref{term1} and from \eqref{term2} has the same sign factor
$$
(-1)^{\ve+ \ve'}, 
$$
where 
$$
\ve = |\ga_{\al}| (1 + |f_1| + \dots + |f_m|) + |v_1^{\al}| + \dots + |v_{i-1}^{\al}| + 
|f_{1}| + \dots + |f_i|\,,
$$
and 
$$
\ve' = |v_1^{\al}|(|f_2|+\dots +|f_{m}|) + \dots + |v^{\al}_{m-1}| |f_{m}|\,.
$$

Proposition \ref{prop:Conv-V-A} is proved.  \qed

\noindent\textsc{Department of Mathematics,
Temple University, \\
Wachman Hall Rm. 638\\
1805 N. Broad St.,\\
Philadelphia PA, 19122 USA \\
\emph{E-mail address:} {\bf vald@temple.edu}}

~\\

\noindent\textsc{Middle College Program \\
Gateway Community College\\ 
New Haven, CT 06510\\
\emph{E-mail address:} {\bf alexhoffnung@gmail.com}}

~\\

\noindent\textsc{Institut f\"ur Mathematik und Informatik\\
Universit\"at Greifswald\\
Walther-Rathenau-Strasse 47\\
17487 Greifswald, Germany
\emph{E-mail address:} {\bf rogersc@uni-greifswald.de} }


\begin{thebibliography}{99}


\bibitem{Berglund} A. Berglund, 
Rational homotopy theory of mapping spaces via Lie theory 
for L-infinity algebras, \href{http://arxiv.org/abs/1110.6145}{arXiv:1110.6145}. 



\bibitem{GMtheorem} V.A. Dolgushev and C.L. Rogers, 
A version of the Goldman-Millson theorem for filtered $L_{\infty}$-algebras, 
\href{http://arxiv.org/abs/1407.6735}{arXiv:1407.6735}.

\bibitem{notes} V.A. Dolgushev and C.L. Rogers, 
Notes on Algebraic Operads, Graph Complexes, and Willwacher's Construction,
{\it Mathematical aspects of quantization,} 25--145, 
Contemp. Math., {\bf 583}, Amer. Math. Soc., Providence, RI, 2012 ;
\href{http://arxiv.org/abs/1202.2937}{arXiv:1202.2937}. 

\bibitem{EnhancedLie} V.A. Dolgushev and C.L. Rogers, 
On an enhancement of the category of shifted $L_{\infty}$-algebras, 
\href{http://arxiv.org/abs/1406.1744}{arXiv:1406.1744}.

\bibitem{RHT} V.A. Dolgushev and C.L. Rogers, The spatial realization 
functor revisited, in preparation.

\bibitem{DeligneTw} V. Dolgushev and T. Willwacher, 
Operadic Twisting -- with an application to Deligne's conjecture,
accepted to J. Pure and Applied Algebra; \href{http://arxiv.org/abs/1207.2180}{arXiv:1207.2180}.  

\bibitem{DefHomotInvar}  V. Dolgushev and T. Willwacher, 
The deformation complex is a homotopy invariant of a homotopy 
algebra, {\it Developments and Retrospectives in Lie Theory},  Springer, 
Developments in Mathematics, {\bf 38} (2014) 137--158;
\href{http://arxiv.org/abs/1305.4165}{arXiv:1305.4165}. 

\bibitem{DotsPoncin} Vladimir Dotsenko and Norbert Poncin, 
A tale of three homotopies, \href{http://arxiv.org/abs/1208.4695}{arXiv:1208.4695}.  

\bibitem{Doubek} M. Doubek, On resolutions of diagrams of algebras, 
\href{http://arxiv.org/abs/1107.1408}{arXiv:1107.1408}.

\bibitem{Drinfeld-dgcat} V. Drinfeld, DG quotients of DG categories, 
J. Algebra {\bf 272}, 2 (2004) 643--691. 

\bibitem{DK} W. G. Dwyer and D. M. Kan, Function complexes in homotopical algebra, 
Topology {\bf 19}, 4 (1980) 427--440.


\bibitem{FMYau} Y. Fr\'egier, M. Markl, and D. Yau,
The $L_{\infty}$-deformation complex of diagrams of algebras, 
New York J. Math. {\bf 15} (2009) 353--392; \href{http://arxiv.org/abs/0812.2981}{arXiv:0812.2981}. 

\bibitem{Ezra-infty} E. Getzler, Lie theory for 
nilpotent $L_\infty$-algebras, Ann. of Math. (2) {\bf 170}, 1 (2009) 271--301; 
\href{http://arxiv.org/abs/math/0404003}{arXiv:math/0404003}.

\bibitem{GJ} E. Getzler and J.D.S. Jones,
Operads, homotopy algebra and iterated integrals
for double loop spaces, \href{http://arxiv.org/abs/hep-th/9403055}{arXiv:hep-th/9403055}.
 


\bibitem{Hinich:1997} V.\ Hinich, Descent of Deligne groupoids, Internat. Math. Res. Notices {\bf 1997}, 
no.~5, 223--239; \href{http://arxiv.org/abs/alg-geom/9606010}{ arXiv:alg-geom/9606010}.


\bibitem{KhPQ} D. Khudaverdyan, N. Poncin and J. Qiu, 
On the infinity category of homotopy Leibniz algebras, 
Theory Appl. Categ. {\bf 29}, 12 (2014) 332--370;
\href{http://arxiv.org/abs/1308.2583}{arXiv:1308.2583}. 

\bibitem{Kelly} G. M. Kelly, {\it Basic concepts of enriched category theory}, London Math. Soc.
Lect. Note Ser., {\bf 64},  Cambridge Univ. Press, Cambridge, 1982.

\bibitem{Lazarev}  A. Lazarev, 
Maurer-Cartan moduli and models for function spaces,
Adv. Math. {\bf 235} (2013) 296--320; \href{http://arxiv.org/abs/1109.3715}{arXiv:1109.3715}.

\bibitem{LV} J.-L. Loday and B. Vallette,
Algebraic operads, {\it Grundlehren der Mathematischen Wissenschaften}, 
{\bf 346}, Springer, Heidelberg, 2012.


\bibitem{h-alg} J. Lurie, Higher Algebra, a draft is available at 
\href{http://www.math.harvard.edu/~lurie/}{http://www.math.harvard.edu/$\sim$lurie/}.



\bibitem{Markl-h-alg} M. Markl, Homotopy diagrams of algebras,
{\it The 21st Winter School "Geometry and Physics'' (Srni, 2001)}, 
 Rend. Circ. Mat. Palermo (2) Suppl.  {\bf 69} (2002) 161--180;
\href{http://arxiv.org/abs/math/0103052}{arXiv:math/0103052}. 


\bibitem{MVnado1} S. Merkulov and B. Vallette,
Deformation theory of representations of prop(erad)s. I,
J. Reine Angew. Math. {\bf 634} (2009) 51--106.

\bibitem{MVnado11} S. Merkulov and B. Vallette, Deformation theory of representations 
of prop(erad)s. II, J. Reine Angew. Math. {\bf 636} (2009), 123--174.


\bibitem{dgcat} D. Tamarkin, What do DG
categories form?  Compos. Math.  {\bf 143}, 5 (2007) 1335--1358;
\href{http://arxiv.org/abs/math/0606553}{arXiv:math/0606553}.

\bibitem{Bruno-plus-equals} B. Vallette,  
Algebra$+$Homotopy$=$Operad, \href{http://arxiv.org/abs/1202.3245}{arXiv:1202.3245}.

\bibitem{Bruno-HT} B. Vallette, Homotopy theory of homotopy algebras, 
\href{http://arxiv.org/abs/1411.5533}{arXiv:1411.5533}.

\end{thebibliography}
\end{document}